\documentclass[11pt,a4paper,reqno]{amsart}
\usepackage{amsmath, amssymb, latexsym, enumerate,  graphicx,tikz, stmaryrd, eufrak,enumitem}

\usepackage{thmtools}

\usepackage[pdftex]{hyperref}
\usepackage{cleveref}
\usepackage{bbm}
\usepackage{cases}
\usepackage{array}
\usepackage{multirow}
\usepackage{tikz-cd}
\usepackage[all]{xy}

\usepackage{mathtools}

\usepackage[hmarginratio={1:1},vmarginratio={1:1},lmargin=80.0pt,tmargin=90.0pt]{geometry}

\usepackage{dynkin-diagrams}

\usepackage{tikz}
\tikzset{anchorbase/.style={baseline={([yshift=-0.5ex]current bounding box.center)}}}
\usetikzlibrary{decorations.markings}
\usetikzlibrary{decorations.pathreplacing}
\usetikzlibrary{arrows,shapes,positioning,backgrounds}
\tikzstyle directed=[postaction={decorate,decoration={markings,
    mark=at position #1 with {\arrow{>}}}}]
\tikzstyle rdirected=[postaction={decorate,decoration={markings,
    mark=at position #1 with {\arrow{<}}}}]
    
\usepackage{graphicx}
\usepackage{nicefrac}

 \newlength{\baseunit}               
 \newcount{\numlines}                
\newtheorem{thm}{Theorem}

\newtheorem{theorem}[subsubsection]{Theorem}
\newtheorem{lemma}[theorem]{Lemma}
\newtheorem{prop}[theorem]{Proposition}

\newtheorem{conjecture}[theorem]{Conjecture}

\theoremstyle{definition}

\newtheorem{remark}[theorem]{Remark}

\newtheorem{example}[subsubsection]{Example}

\newtheorem{question}[theorem]{Question}

\newcommand{\bk}{\mathbf{k}}

\newcommand{\cA}{\mathcal{A}}
\newcommand{\cB}{\mathcal{B}}
\newcommand{\cC}{\mathcal{C}}
\newcommand{\cD}{\mathcal{D}}
\newcommand{\C}{\mathcal{C}}
\numberwithin{equation}{section}

\newcommand{\ev}{\mathrm{ev}}
\newcommand{\coev}{\mathrm{coev}}
\newcommand{\im}{\mathrm{im}}

\newcommand{\Dim}{\mathrm{Dim}}

\newcommand{\ch}{\mathrm{ch}}

\newcommand{\charr}{\mathrm{char}}

\newcommand{\Ver}{\mathsf{Ver}}

\newcommand{\Rep}{\mathsf{Rep}}

\newcommand{\unit}{{\mathbbm{1}}}
\newcommand{\tto}{\twoheadrightarrow}
\newcommand{\cO}{\mathcal{O}}

\newcommand{\mN}{\mathbb{N}}
\newcommand{\mZ}{\mathbb{Z}}

\newcommand{\mC}{\mathbb{C}}

\newcommand{\mQ}{\mathbb{Q}}
\newcommand{\fg}{\mathfrak{g}}

\newcommand{\End}{\mathrm{End}}

\newcommand{\Hom}{\mathrm{Hom}}

\newcommand{\Sym}{\mathrm{Sym}}
\newcommand{\id}{\mathrm{id}}

\newcommand{\Ind}{\mathrm{Ind}}

\newcommand{\Vecc}{\mathsf{Vec}}
\newcommand{\sVecc}{\mathsf{sVec}}
\newcommand{\sVec}{\mathsf{sVec}}

\newcommand{\Tilt}{\mathsf{Tilt}}

\newcommand{\mF}{\mathbb{F}}

\newcommand{\Tr}{\mathrm{Tr}}

\newcommand{\ad}{\mathrm{ad}}

\newcommand{\gr}{\mathrm{gr}}

\newcommand{\GL}{\mathrm{GL}}
\newcommand{\PGL}{\mathrm{PGL}}
\newcommand{\SL}{\mathrm{SL}}

\newcommand{\PO}{\mathrm{PO}}
\newcommand{\SO}{\mathrm{SO}}
\newcommand{\Spin}{\mathrm{Spin}}
\newcommand{\PSO}{\mathrm{PSO}}
\newcommand{\Sp}{\mathrm{Sp}}

\newcommand{\SOSp}{\mathrm{SOSp}}

\begin{document}
\title{Finite symmetric algebras in tensor categories and Verlinde categories of algebraic groups}

\author{Kevin Coulembier}
\address{School of Mathematics and Statistics, University of Sydney, Australia}
\email{kevin.coulembier@sydney.edu.au}
\author{Pavel Etingof}
\address{Department of Mathematics, MIT, Cambridge, MA USA 02139
}\email{etingof@math.mit.edu}
\author{Joseph Newton}
\address{School of Mathematics and Statistics, University of Sydney, Australia}
\email{j.newton@maths.usyd.edu.au}

\date{\today}

\keywords{}

\setcounter{tocdepth}{2}

\begin{abstract}
We investigate objects in symmetric tensor categories that have simultaneously finite symmetric and finite exterior algebra. This forces the characteristic of the base field to be $p>0$, and the maximal degree of non-vanishing symmetric and exterior powers to add up to a multiple of $p$. We give a complete classification of objects in tensor categories for which this sum equals $p$. All resulting tensor categories are Verlinde categories of reductive groups and we fill in some gaps in the literature on these categories. 
\end{abstract}

\maketitle

\tableofcontents

\section*{Introduction}

In recent years, considerable progress has been made in the study of (symmetric) tensor categories of moderate growth over fields of positive characteristic, see \cite{BE1, BE, BEO, CE, CEO1, CEO3, Co, Et, EO, Os, Ve1} and references therein. In characteristic zero, the theory is well-established by \cite{Del01, Del02}. In characteristic zero, it is also clear that a non-zero object cannot both have a vanishing symmetric and a vanishing exterior power. In contrast, such objects seem to play an essential role in positive characteristic. 

Concretely, by \cite{CEO3}, the structure theory of tensor categories can be reduced to the classification of `incompressible' tensor categories and the question of which tensor categories fibre over which incompresible categories. Moreover, the combination of \cite[Conjecture~1.4]{BEO} and \cite[Conjecture~3.2.3]{ComAlg} predicts that any non-invertible simple object in an incompressible tensor category has finite symmetric and exterior algebras. Additionally, it was proved in \cite{CE} that, under mild assumptions, an object $X$ with vanishing third exterior power has vanishing symmetric power in degree $N-1$ (and not below) if and only if the endofunctor $-\otimes X$ is $N$-spherical in the sense of \cite{DKS}.

Motivated by these considerations, we initiate a systematic study of objects in tensor categories with finite symmetric and finite exterior algebra. In particular, we focus on the extremal case where the algebras are as small as theoretically possible (the highest non-zero degrees add up to $p$). Our main result is the following theorem concerning tensor categories over an algebraically closed field $\bk$ of characteristic $p>0$:

\begin{thm}\label{Thmmnp}
Let $\cC$ be a tensor category over $\bk$ generated by a non-invertible object $X\in\cC$, and suppose there are integers $m,n\geq2$ with $m+n=p\geq5$, $\Lambda^nX\not=0\not=\Sym^mX$, and $\Lambda^{n+1}X=0=\Sym^{m+1}X$. Then $\cC$ is semisimple, and $\cC$ and $X$ are as follows:
\begin{enumerate}
\item If $\Lambda^nX\cong\unit\cong\Sym^mX$ then $n$ is odd, and we have one of the following cases:
\begin{enumerate}
\item $\cC\simeq\Ver_p(\PGL_n)\simeq\Ver_p^+(\SL_n)$;
\item $\cC\simeq\Ver_p(\SO_n)\simeq\Ver_p^+(\Spin_n)$;
\item $\cC\simeq\Ver_p^+\simeq\Ver_p(\PGL_2)$;
\item $n=7$, $p\geq13$ and $\cC\simeq\Ver_p(G_2)$;
\item $p=2n+1$ or $p=2n-1$, and $\cC\simeq\Ver_p(\PSO_{p-1})\simeq\Ver_p^+(\Spin_{p-1})$;
\item $n=13$, $p=23$ and $\cC\simeq\Ver_{23}^+(E_7)$.
\end{enumerate}
The images of $X$ under each of the equivalences above are listed in Theorem~\ref{Thmmnp1}.
\item More generally, we have $\cC\simeq\cC^+\boxtimes\cC_0$ where $\cC^+$ is a subcategory of $\cC$ equivalent to one of the categories above with generator $X^+$, $\cC_0$ is a pointed semisimple subcategory generated by an invertible object $X_0$, and $X\cong X^+\otimes X_0$.
\end{enumerate}
\end{thm}

Here $\Ver_p(G)$, for a reductive group $G$ over $\bk$, is the semisimplification of the category of tilting modules of $G$, first introduced in \cite{GK, GM}. Such categories play a central role among semisimple tensor categories in positive characteristic, see \cite{CEO1, CEO2, Os}. In order to establish our main theorem, we prove a number of facts about these tensor categories, mainly concerning the invertible objects in $\Ver_p(G)$ and the connection between $\Ver_p(G)$ and representations of the image of the Lie algebra of $G$ inside $\Ver_p=\Ver_p(\SL_2)$, by taking a principal $\SL_2\to G$. Many of these statements are known, but not always with published proofs. Additionally, as a consequence of the methods behind the main theorem, we classify all possible equivalences $\Ver_p(G)\simeq\Ver_p(G')$ for distinct adjoint-type simple algebraic groups~$G,G'$.

Theorem~\ref{Thmmnp}, and in fact already the `base case' (1)(c), show that the potential finite values of $m+n$ are precisely the multiples of $p$ when $p\ge 5$. For completeness we prove, in Proposition~\ref{PropNX6}, that the potential values for $p\in\{2,3\}$ are precisely the multiples of $p^2$.

The paper is organised as follows. In Section~\ref{SecPrelim} we recall the relevant notions concerning tensor categories. In Section~\ref{SecFiniteP} we establish some basic properties of symmetric and exterior powers, mainly relating to filtrations. These results are new if the characteristic of the base field is $2$. In Section~\ref{SecVerpG} we establish the necessary results on $\Ver_p(G)$. In Section~\ref{SecMainThm} we prove Theorem~\ref{Thmmnp} by reducing it to a classification of Lie subalgebras of special linear Lie algebras in $\Ver_p$ that we solve computationally. In Section~\ref{SecExtra} we collect some computational results on symmetric and exterior powers of simple objects in some small non-semisimple incompressible categories, leading to a refinement of the conjecture regarding vanishing powers, and focus on $p\in\{2,3\}$.

\section{Preliminaries}\label{SecPrelim}

\subsection{Tensor categories} 

\subsubsection{} Let $\bk$ be an algebraically closed field. We write `tensor category' for a symmetric rigid monoidal $\bk$-linear abelian category with unit object $\unit$, where $\End(\unit)\cong\bk$ and all objects have finite length. We write `tensor functor' for a symmetric monoidal exact functor between tensor categories. Note that such categories and functors are often labelled `symmetric' in the literature, for instance in \cite{EGNO}, but we omit the word `symmetric' for brevity. We write $c_{X,Y}$ for the braiding morphism $X\otimes Y\to Y\otimes X$, and we occasionally write just $c$ and omit the objects if they are clear from context. We say that a tensor category $\cC$ is generated by an object $X\in\cC$ if every object in $\cC$ can be obtained from $X$ via taking tensor products, duals, direct sums and subquotients. We write $\Vecc$ and $\sVec$ for the categories of finite-dimensional vector spaces and super vector spaces respectively, and write $\bar\unit$ for the 1-dimensional odd super vector space.

\subsubsection{} For a tensor category $\cC$, let $\pi(\cC)$ be the fundamental group of $\cC$ as introduced in \cite{Del01}, see also \cite[\S 4.1]{CEO3}. For an affine group scheme $G$ in the category $\cC$, we say that a group morphism $\phi:\pi(\cC)\to G$ is a constraint on $G$ if $(G,\phi)$ is a $\cC$-group in the sense of \cite{CEO3}, that is the canonical action of $\pi(\cC)$ on $\cO(G)$ equals the adjoint action via $\phi$. We write $\Rep(G,\phi)$ for the category of representations of $G$ for which the two actions of $\pi(\cC)$ (one the canonical action in $\cC$ and the other coming from the action of $G$ via $\phi$) are the same. Similarly, if $\fg$ is a Lie algebra in $\cC$ and $\fg_\pi$ is the Lie algebra of $\pi(\cC)$, then we say that a Lie algebra morphism $\phi:\fg_\pi\to\fg$ is a constraint on $\fg$ if the canonical action of $\fg_\pi$ on $\fg$ equals the adjoint action via $\phi$. We write $\Rep(\fg,\phi)$ for the category of representations of $\fg$ for which the two actions of $\fg_\pi$ (coming from $\fg$ and $\pi(\cC)$) are the same.

\subsection{Invertible objects}
\subsubsection{} Let $\cC$ be a tensor category over $\bk$. The categorical dimension $\dim(X)$ of an object $X$ in $\cC$ is the unique element of $\bk$ such that the morphism $\Tr(\id_X):\unit\to\unit$ defined as the composition
$$\begin{tikzcd}
\unit \arrow{r}{\coev_X} & X\otimes X^* \arrow{r}{c_{X,X^*}} & X^*\otimes X \arrow{r}{\ev_X} & \unit
\end{tikzcd}$$
is equal to $\dim(X)\cdot\id_\unit$. We have $\dim(X\otimes Y)=\dim(X)\dim(Y)$ for all $X,Y\in\cC$.

We call an object $L$ in $\cC$ invertible if $L^*\otimes L\cong\unit$. An invertible $L$ is necessarily simple, and if $X$ is a simple object in $\cC$ then $L\otimes X$ must also be simple, since $L^*\otimes L\otimes X\cong X$. The braiding morphism $c_{L,L}:L\otimes L\to L\otimes L$ on an invertible object $L$ must satisfy $c_{L,L}^2=\id_{L\otimes L}$, so for $\charr(\bk)\neq2$ we say $L$ is even if $c_{L,L}=\id_{L\otimes L}$ and $L$ is odd if $c_{L,L}=-\id_{L\otimes L}$.

\begin{lemma}\label{LemInvDim}
If $L$ is invertible and $\charr(\bk)\neq2$, then $L$ is even if and only if $\dim(L)=1$ and odd if and only if $\dim(L)=-1$.
\end{lemma}

\begin{proof}
By definition of evaluation and coevaluation morphisms, the composition
\[L\xrightarrow{\coev_L\otimes\id}L\otimes L^*\otimes L\xrightarrow{c_{L,L^*}\otimes\id}L^*\otimes L\otimes L\xrightarrow{\id\otimes c_{L,L}}L^*\otimes L\otimes L\xrightarrow{\ev_L\otimes\id}L\]
must be $\id_L$. If $c_{L,L}=\pm\id_{L\otimes L}$ then this composition equals $\pm\Tr(\id_X)\otimes\id_L$, giving the result.
\end{proof}

%

\subsubsection{} We call $\cC$ pointed if every simple object in $\cC$ is invertible. For an abelian group $\Gamma$ we define $\Vecc_\Gamma$ to be the semisimple pointed tannakian category with all simples 1-dimensional and indexed by elements of $\Gamma$ so that $g\otimes h=gh$ for $g,h\in\Gamma$. For a group homomorphism $\phi:\Gamma\to\mZ/2$, we define $\sVec_{\Gamma,\phi}$ to be the semisimple pointed supertannakian category with underlying tensor category $\Vecc_\Gamma$ and each simple $g$ being even if $\phi(g)=0$ and odd if $\phi(g)=1$.

If $\cC$ is pointed and semisimple then it is supertannakian by \cite[6.2.3]{Co}, and thus we have the following well-known fact.

\begin{lemma}\label{LemPointed} If $\cC$ is a pointed semisimple tensor category, then it is equivalent to $\sVecc_{\Gamma,\phi}$, where $\Gamma$ is the group of isomorphism classes of invertible objects with operation $\otimes$, and the map $\phi:\Gamma\to\mZ/2$ sends even invertibles to 0 and odd invertibles to 1.
\end{lemma}

We will often need to decompose a tensor category into its invertible and non-invertible objects, so we establish the following lemma.

\begin{lemma}\label{LemDecomp}
Suppose $\cA$ and $\cB$ are tensor subcategories of $\cC$ such that $\cB$ is pointed and semisimple, the only invertible object in $\cA$ is $\unit$, and $\cC$ is tensor-generated by $\cA$ and $\cB$ together. Then the tensor functor $\cA\boxtimes\cB\to\cC$ sending $A\boxtimes B$ to $A\otimes B$ is an equivalence.
\end{lemma}

\begin{proof}
We have a surjective tensor functor $F:\cA\boxtimes\cB\to\cC$, and by \cite[\S 3.1]{CEO3} it suffices to show $F$ sends simples to simples and is fully faithful. The former holds since tensoring by invertibles sends simples to simples. For the latter we require
$$\Hom(A_1,A_2)\otimes_\bk\Hom(B_1,B_2)\cong\Hom(A_1\otimes B_1,A_2\otimes B_2)$$
for all $A_1,A_2\in\cA$ and $B_1,B_2\in\cB$. This reduces to
$$\Hom(A,\unit)\otimes_\bk\Hom(\unit,B)\cong\Hom(A,B)$$
for all $A,B$ indecomposable, using duality adjunctions and extrapolating from indecomposables to their direct sums. If $B\cong\unit$ then this clearly holds, and otherwise we must have $\Hom(A,B)=0$ or else the invertible $B$ is in $\cA$, so it holds in all cases.
\end{proof}

\section{Symmetric and exterior powers}\label{SecFiniteP}

Let $\cC$ be a tensor category over an algebraically closed field $\bk$ with $p=\mathrm{char}(\bk)>0$.

\subsection{Definitions}

\subsubsection{}  If $X$ is an object in $\cC$, then the second exterior power is the image of $\mathrm{id}_{X\otimes X}-c_{X,X}$, which we denote by $\Lambda^2X$ (resp. $\wedge^2X$) when treated as a subobject (resp. quotient) of $X^{\otimes2}$. We define $\Sym^2X$ and $\Gamma^2X$ by the exact sequences
$$0\to\Lambda^2X\to X\otimes X\to\Sym^2X\to0, \qquad\qquad 0\to\Gamma^2X\to X\otimes X\to\wedge^2X\to0.$$
The symmetric and exterior powers for $n\geq3$ are defined by the exact sequences
$$\begin{tikzcd}[row sep=0.5em]
0 \arrow{r} & \displaystyle\sum_{i=1}^{n-1}X^{\otimes i-1}\otimes\Lambda^2X\otimes X^{\otimes n-i-1} \arrow{r} & X^{\otimes n} \arrow{r} & \Sym^nX \arrow{r} & 0\\
0 \arrow{r} & \displaystyle\sum_{i=1}^{n-1}X^{\otimes i-1}\otimes\Gamma^2X\otimes X^{\otimes n-i-1} \arrow{r} & X^{\otimes n} \arrow{r} & \wedge^nX \arrow{r} & 0
\end{tikzcd}$$
and their dual variants, the divided and `dual exterior' powers, are given by
\begin{align*}
\Gamma^nX&=\bigcap_{i=1}^{n-1}X^{\otimes i-1}\otimes\Gamma^2X\otimes X^{\otimes n-i-1},\\
\Lambda^nX&=\bigcap_{i=1}^{n-1}X^{\otimes i-1}\otimes\Lambda^2X\otimes X^{\otimes n-i-1}.
\end{align*}

\begin{remark}\label{RemSymSplit}
$\Sym^nX$ can alternately be defined as the cokernel of the morphism
\[\bigoplus_{i=1}^{n-1}\id_X^{\otimes i-1}\otimes(\id_{X\otimes X}-c_{X,X})\otimes\id_X^{\otimes n-i-1}:(X^{\otimes n})^{n-1}\to X^{\otimes n}\]
and $\Gamma^nX$ is the kernel of a similar morphism $X^{\otimes n}\to (X^{\otimes n})^{n-1}$. In characteristic $p\geq3$, both $\Sym^2X$ and $\Gamma^2X$ are isomorphic to the image of $\id_{X\otimes X}+c_{X,X}$ and we have a splitting $X\otimes X\cong\wedge^2X\oplus\Sym^2X$, meaning $\wedge^nX$ is the quotient of the same morphism above except with $\id_{X\otimes X}-c_{X,X}$ replaced by $\id_{X\otimes X}+c_{X,X}$. In characteristic 2 however, we have $(\id_{X\otimes X}-c_{X,X})^2=0$ meaning this splitting does not exist and we cannot necessarily write $\wedge^nX$ as a quotient of a morphism $(X^{\otimes n})^{n-1}\to X^{\otimes n}$. In fact, $\Lambda^nX\subseteq\Gamma^nX$ and so $\wedge^nX$ is a quotient of $\Sym^nX$ for all $n\in\mN$ when $p=2$.
\end{remark}

\begin{lemma}\label{LemExtDims}
Let $s_n=\sum_{\sigma\in S_n}\sigma$ and $a_n=\sum_{\sigma\in S_n}\varepsilon(\sigma)\sigma$ be the symmetrizer and antisymmetrizer on $X^{\otimes n}$, where $\varepsilon:S_n\to\{1,-1\}$ is the sign homomorphism. If $n<p$ then
$$\Sym^nX\cong\im(s_n)\cong\Gamma^nX,\qquad \wedge^nX\cong\im(a_n)\cong\Lambda^nX$$
where $\im$ means the image of a morphism. We also have
$$\dim(\im(s_n))={\dim X+n-1 \choose n},\qquad\dim(\im(a_n))={\dim X\choose n}$$
and for $n<p$ these formulas also apply to $\Sym,\Gamma$ and $\wedge,\Lambda$ respectively.
\end{lemma}

\begin{proof}
For $p=2$ we have $n\in\{0,1\}$ and the lemma is obvious. If $n<p\geq3$ then the morphisms $s_n/n!$ and $a_n/n!$ are idempotent, so their images are summands of $X^{\otimes n}$. Then $s_n/n!$ is annihilated by $1-\sigma$ for all $\sigma\in S_n$, and since the projection $f:X^{\otimes n}\twoheadrightarrow\Sym^nX$ satisfies $f=f\sigma$, $f$ annihilates the complement $1-s_n/n!$. Hence $\Sym^nX\cong\im(s_n/n!)$, and since $p\geq3$ we can repeat for $\wedge^nX$ replacing $\sigma$ by $\varepsilon(\sigma)\sigma$. $\Gamma^nX$ and $\Lambda^nX$ follow dually.

The dimension formulas are derived for characteristic 0 in \cite[7.1]{Del01}, and the same proof applies in characteristic $p$ when $n<p$.
\end{proof}

\begin{remark}
If $L\in\cC$ is invertible then $1-c_{L,L}=0$ if $L$ is even and $2$ if $L$ is odd. Consequently, if $L$ is even (or if $p=2$) then $\Sym^nL\cong L^{\otimes n}$ for all $n\in\mN$ while $\Lambda^nL=0$ for $n\geq2$, and these are reversed for odd $L$ when $p\neq2$.
\end{remark}

\subsection{Symmetric and exterior powers of filtered objects}\label{SecFiltObjects}

\subsubsection{} We write $TX=\bigoplus_{n=0}^\infty X^{\otimes n}\in\Ind\cC$ for the tensor algebra on $X$, which is a graded Hopf algebra with comultiplication given by the diagonal morphism $X\to X\otimes\unit\oplus\unit\otimes X$ on the first component. The unit and counit are inclusion of and projection onto $X^{\otimes0}=\unit$ respectively, and the antipode is given by $(-1)^n$ on $X^{\otimes n}$. We can then write $\Sym X=\bigoplus_{n=0}^\infty\Sym^nX$ and $\wedge X=\bigoplus_{n=0}^\infty\wedge^nX$ as quotients of $TX$ by (co)ideals,
\begin{align*}
\Sym X&=TX\Big/\sum_{i,j\in\mN}X^{\otimes i}\otimes\Lambda^2X\otimes X^{\otimes j},\\
\wedge X&=TX\Big/\sum_{i,j\in\mN}X^{\otimes i}\otimes\Gamma^2X\otimes X^{\otimes j}.
\end{align*}
This induces Hopf algebra structures on $\Sym X$ and $\wedge X$.

Now suppose $0\to X\to Y\to Z\to 0$ is an exact sequence. This induces a filtration on $Y^{\otimes n}$ with composition factors $\bigoplus_\sigma\sigma(X^{\otimes n-i}\otimes Z^{\otimes i})$ for each $0\leq i\leq n$, where the direct sum is over all shuffle permutations $\sigma$ of the tensor product. This in turn induces filtrations on subobjects and quotients of $Y^{\otimes n}$, and we use the prefix $\gr$ to denote the graded object corresponding to a filtration, meaning the direct sum of the composition factors.

\begin{lemma}\label{LemGradedComp}
For an exact sequence $0\to X\to Y\to Z\to 0$ in $\cC$, we have
\begin{equation*}
\begin{array}{ccccccc}
\gr Y^{\otimes2} &\cong& X\otimes X &\oplus& (X\otimes Z\oplus Z\otimes X) &\oplus& Z\otimes Z\\
\gr\Gamma^2Y &\cong& \Gamma^2X &\oplus& X\otimes Z &\oplus& \im(\Gamma^2Y\to\Gamma^2Z)\\
\gr\Lambda^2Y &\cong& \Lambda^2Y\cap(X\otimes X) &\oplus& X\otimes Z &\oplus& \Lambda^2Z
\end{array}
\end{equation*}
where the three columns are the three graded components. The components of $\gr\Sym^2Y$ and $\gr\wedge^2Y$ are computed dually to those of $\gr\Gamma^2Y$ and $\gr\Lambda^2Y$ respectively.
\end{lemma}

\begin{proof}
Consider the following commutative diagram:
$$\begin{tikzcd}
\Gamma^2Y\cap(X\otimes X) \arrow[hook]{r} \arrow[hook]{d} & K\coloneqq\Gamma^2Y\cap(X\otimes Y+Y\otimes X) \arrow[hook]{d} \arrow[hook]{r} & \Gamma^2Y \arrow[hook]{d}\\
X\otimes X \arrow[hook]{r} \arrow{d}{1-c} & X\otimes Y+Y\otimes X \arrow[hook]{r} \arrow{d}{1-c} & Y\otimes Y \arrow{d}{1-c}\\
X\otimes X \arrow[hook]{r} & X\otimes Y+Y\otimes X \arrow[hook]{r} & Y\otimes Y
\end{tikzcd}$$ 
The top two rows are the induced filtrations on $\Gamma^2Y$ and $Y^{\otimes2}$ respectively, and each arrow labelled $1-c$ is the restriction of $\id_{Y\otimes Y}-c_{Y,Y}$ to each subobject. Since the top two squares are both pullbacks we have that $K$ is the kernel of $1-c$ on $X\otimes Y+Y\otimes X$ and $\Gamma^2Y\cap(X\otimes X)=\Gamma^2X$.

Now fix an exact faithful $\bk$-linear functor $\cC \to\Vecc$ (that we leave out of the notation) so that the objects $X\otimes Y+Y\otimes X$ and
$$(X\otimes Y+Y\otimes X)/(X\otimes X)\cong X\otimes Z\oplus Z\otimes X$$
endowed with the braiding action are both $\mZ/2$-representations. Then in $\Rep(\mZ/2)$ the object $X\otimes Z\oplus Z\otimes X\cong(\bk\mZ/2)^{\dim_\bk(X\otimes Z)}$ is projective, and so $X\otimes Y+Y\otimes X\tto X\otimes Z\oplus Z\otimes X$ is split. Hence, the map from $\mZ/2$-invariants in $X\otimes Y+Y\otimes X$ to invariants in $X\otimes Z\oplus Z\otimes X$ is an epimorphism. But this means the middle graded component of $\Gamma^2Y$, which is given by
$$\im(K\hookrightarrow X\otimes Y+Y\otimes X\twoheadrightarrow X\otimes Z\oplus Z\otimes X),$$
is exactly the kernel of $1-c$ on $X\otimes Z\oplus Z\otimes X$ which is isomorphic to $X\oplus Z$.

The third graded component of $\Gamma^2Y$ is $\im(\Gamma^2Y\to\Gamma^2Z)$ by definition, and the components of $\wedge^2Y$ (and dually $\Lambda^2Y$) follow from the exact sequence $0\to\Gamma^2Y\to Y^{\otimes2}\to\wedge^2Y\to0$.
\end{proof}

\subsubsection{}\label{SecBigrading} Since we have $\gr(Y^{\otimes n})=(X\oplus Z)^{\otimes n}$, there is a bigrading (i.e. $\mN^2$-grading) on
$$\gr TY\coloneqq\bigoplus_{i=1}^n\gr(Y^{\otimes n})=T(X\oplus Z)$$
where permutations of $X^{\otimes i}\otimes Z^{\otimes j}$ have index $(i,j)\in\mN^2$, so the summand $\gr(Y^{\otimes n})\subset\gr TY$ is comprised of the graded components with indices $(n-i,i)$ for $0\leq i\leq n$. The Hopf algebra structure on $T(X\oplus Z)$ respects the bigrading, that is the multiplication and comultiplication restrict to components as
\begin{align*}
T(X\oplus Z)[i,j]&\otimes T(X\oplus Z)[k,l]\xrightarrow{\text{mult.}} T(X\oplus Z)[i+j,k+l],\\
T(X\oplus Z)[i,j]&\xrightarrow{\text{comult.}}\bigoplus_{0\leq k\leq i,0\leq l\leq j}T(X\oplus Z)[i-k,j-l]\otimes T(X\oplus Z)[k,l]
\end{align*}
respectively. Similarly, $\gr\Sym Y\coloneqq\bigoplus_{n=0}^\infty\gr\Sym^nY$ and $\gr\wedge Y\coloneqq\bigoplus_{n=0}^\infty\gr\wedge^nY$ are bigraded with $(i,j)$-component being the $j$-th component of $\gr\Sym^{i+j}Y$ and $\gr\wedge^{i+j}Y$ respectively. They are also bigraded Hopf algebras, since they are quotients of $T(X\oplus Z)$ by the graded versions of the subobjects
$$I_\Lambda=\sum_{i,j}Y^{\otimes i}\otimes\Lambda^2Y\otimes Y^{\otimes j},\quad
I_\Gamma=\sum_{i,j}Y^{\otimes i}\otimes\Gamma^2Y\otimes Y^{\otimes j}$$
respectively, which it can be checked are (co)ideals of $T(X\otimes Y)$. We recall the following lemma from \cite{CEO3}:

\begin{lemma}\label{LemAlgIdeals}\samepage
\cite[Lemma~3.3.3]{CEO3} Suppose $A,B$ are graded coalgebras in $\Ind\cC$ whose counits are projections onto the zeroth components $A[0]\cong\unit\cong B[0]$, meaning $A\cong A\otimes B[0]$ and $B\cong A[0]\otimes B$ are canonically subcoalgebras of $A\otimes B$. If $I\subseteq A\otimes B$ is a bigraded coideal, then
$$I=(I\cap A)\otimes B+A\otimes(B\cap I),\text{ and }(A\otimes B)/I\cong(A/(I\cap A))\otimes(B/(I\cap B))$$
as bigraded coalgebras.
\end{lemma}

Using the bigrading on $T(X\otimes Y)$ and Lemmas~\ref{LemGradedComp} and \ref{LemAlgIdeals}, we can make the following extension of \cite[Theorem~3.3.2]{CEO3}:

\begin{theorem}\label{ThmGradedAlgs}
If $0\to X\to Y\to Z\to0$ is an exact sequence, then the induced graded objects $\gr\Sym Y$ and $\gr\wedge Y$ are both bigraded Hopf algebras given by
$$\gr\Sym Y\cong R_0\otimes\Sym Z,\qquad\qquad\gr\wedge Y\cong R_1\otimes R_2,$$
where $R_0$ is the image of $\Sym X\to\Sym Y$, $R_1$ is the image of $\wedge X\to\wedge Y$ and $R_2$ is the pushout of the quotients $TY\twoheadrightarrow\wedge Y$ and $TY\twoheadrightarrow TZ$. Moreover, for $R\in\{R_0,\Sym Z,R_1,R_2\}$ we have $R[0]=\unit$, $R[1]=X$ or $Z$, and $R[i]=0\implies R[i+1]=0$ for all $i$.
\end{theorem}

\begin{proof}
Using the middle graded components of $\Lambda^2Y$ and $\Gamma^2Y$ found in Lemma~\ref{LemGradedComp}, we have inclusions of
$$J\coloneqq\sum_{i,j}(X\oplus Z)^{\otimes i}\otimes(X\otimes Z)\otimes(X\oplus Z)^{\otimes j}$$
into $\gr I_\Lambda$ and $\gr I_\Gamma$, each of which expresses $J$ as a (co)ideal of $T(X\oplus Z)$. In both cases the quotient of $T(X\oplus Z)$ by $J$ is the Hopf algebra $TX\otimes TZ$ with bigrading assigning $X^{\otimes i}\otimes Z^{\otimes j}$ the index $(i,j)$, since each quotient identifies $X\otimes Z$ and $Z\otimes X$ and similarly identifies all permutations of $X^{\otimes i}\otimes Z^{\otimes j}$. This means we can express $\gr\Sym Y$ and $\gr\wedge Y$ as quotients of $TX\otimes TY$, and then applying Lemma~\ref{LemAlgIdeals} gives $\gr\wedge Y=R_1\otimes R_2$, where $R_1$ and $R_2$ are graded algebras with components $R_1[n]$ and $R_2[n]$ given by the first and last components of $\gr\wedge^nY$ respectively for each $n\in\mN$. Thus $R_1$ is the image of $TX\to TY\to\wedge Y$ which factors through $\wedge X$, and $R_2$ is the pushout of $\wedge Y$ and $TZ$.

The same applies to the symmetric algebra, except that the pushout of $\Sym Y$ and $TZ$ is exactly $\Sym Z$ since we have the following commutative diagram with exact rows:
$$\begin{tikzcd}
(Y^{\otimes n})^{n-1} \arrow{r}{\bigoplus_{i=1}^{n-1}\id_Y^{\otimes i-1}\otimes(1-c)\otimes\id_Y^{\otimes n-i-1}} \arrow[two heads]{d} &[10em]
Y^{\otimes n} \arrow[two heads]{r} \arrow[two heads]{d} & \Sym^nY \arrow{r} \arrow[two heads]{d} & 0\\
(Z^{\otimes n})^{n-1} \arrow{r}{\bigoplus_{i=1}^{n-1}\id_Z^{\otimes i-1}\otimes(1-c)\otimes\id_Z^{\otimes n-i-1}} &
Z^{\otimes n} \arrow[two heads]{r} & \Sym^nZ \arrow{r} & 0
\end{tikzcd}$$
\end{proof}

\begin{remark}
In Theorem~\ref{ThmGradedAlgs} and Lemma~\ref{LemGradedComp} there is an asymmetry between the behaviours of the symmetric and exterior powers, but this is only relevant in characteristic 2. For $\mathrm{char}(\bk)\neq2$ the splitting in Remark~\ref{RemSymSplit} gives us $R_2=\wedge Z$ and $\Lambda^2Y\cap(X\otimes X)=\Lambda^2X$. A counterexample to both of these in characteristic 2 is to take $Y$ to be the length 2 projective object in $\Ver_4$ (see \cite{BEO}) with composition factors $X=\unit$ and $Z=\unit$.
\end{remark}

\subsection{Objects with finite powers}

\subsubsection{} Motivated by \cite[Proposition~6.6.2]{CE}, for $X\in\cC$ we set
\begin{align*}
m(X)&\coloneqq\sup\{m\in\mN\mid \Sym^m X\not=0\}\;\in\;\mN\cup\{\infty\},\\
n(X)&\coloneqq\sup\{n\in\mN\mid \Lambda^nX\not=0\}\;\in\;\mN\cup\{\infty\},\\
\mathbf{N}(X)&\coloneqq m(X)+n(X)\in\mN\cup\{\infty\}.
\end{align*}

For example, $X=0$ if and only if $\mathbf{N}(X)=0$, and $X$ is invertible if and only if we have $\min(m(X),n(X))=1$. Note that we use $\Lambda$ instead of $\wedge$ to define $n(X)$, and if we replace $\Sym$ by $\Gamma$ and $\Lambda$ by $\wedge$ then we obtain $m(X^*)$ and $n(X^*)$. These are not guaranteed to be equal to $m(X)$ and $n(X)$, as seen in \cite[\S 10.2]{CEO3}.

We can alternately characterise both $m(X)$ and $n(X)$ as follows.

\begin{prop}\label{LemSymInv}
Take $0\not=X\in\cC$ which is not invertible. There is at most one $m\in\mZ_{>0}$ for which $\Sym^mX$ is invertible. If such $m$ exists, then $m=m(X)$; if such $m$ does not exist, then $m(X)=\infty$. Similarly, there is at most one $n\in\mZ_{>0}$ for which $\Lambda^nX$ is invertible. If such $n$ exists, then $n=n(X)$; if such $n$ does not exist, then $n(X)=\infty$.
\end{prop}

Proposition~\ref{LemSymInv} is proven in \cite{CE} for symmetric powers, and it is observed that the same result holds for exterior powers by considering $X^\ast\boxtimes\bar{\unit}$ in $\cC\boxtimes\sVec$. However this assumes $p\geq3$, so we will provide a more careful justification below which also applies when $p=2$.

\begin{proof}[Proof of Proposition~\ref{LemSymInv}]
That $\Sym X$ finite implies $\Sym^nX$ is invertible for the maximal $n$ with $\Sym^nX\neq0$ is proven in \cite[Lemma 7.4.5]{CEO3} and the converse is proven in \cite[Corollary 7.4.4]{CEO3} for $X$ simple, and both of these results immediately apply to $\Lambda X$ as well. The latter result is extended to non-simple objects in \cite[Lemma 6.1.2]{CE}, and we restate this proof and generalise it to apply to $\wedge X$.

Suppose $Y$ is non-simple, so there is an exact sequence $0\to X\to Y\to Z\to 0$. Since invertible objects are simple, if $\wedge^nY$ is invertible then exactly one graded component of $\gr\wedge^nY=R_1\otimes R_2$ is non-zero. We have $R_1[0]\cong R_2[0]\cong\unit$, $R_1[1]\cong X$ and $R_2[1]\cong Z$, and using the fact that $\wedge^{n+1}Y$ is the pushout of $\wedge^nY\otimes Y$ and $Y\otimes\wedge^nY$ we have $R_j[i]=0\implies R_j[i+1]=0$ for all $i\in\mN$, $j\in\{1,2\}$. Consequently, if
$$\gr\wedge^nY=(R_1\otimes R_2)[n]=\bigoplus_{i=0}^nR_1[i]\otimes R_2[n-i]$$
has only one non-zero component then there must be some $1\leq j\leq n-1$ with $R_1[i]=0\iff i>j$ and $R_2[i]=0\iff i>n-j$. This implies $(R_1\otimes R_2)[n+1]=0$, so $\gr\wedge^{n+1}Y=0$ and hence $\wedge^{n+1}Y=0$. The same argument applies to $\Sym Y$, and by considering $\wedge(Y^*)$ it also applies to the powers $\Lambda^nY$.
\end{proof}

We also state the following two lemmas which will be useful throughout this paper. 

\begin{lemma}[Lemma 6.2.6 in \cite{CE}]\label{LemDim}
For every $0\not=X$ we have $\mathbf{N}(X)\ge p$. In fact, $\mathbf{N}(X)$ is a multiple of $p$.
\end{lemma}

\begin{lemma}\label{LemEven}
If $X$ is not invertible and $n=n(X)$ satisfies $2\leq n\leq p-2$, then the invertible object $\Lambda^nX$ is even.
\end{lemma}

\begin{proof}
By Lemma~\ref{LemExtDims}, we have
\begin{align*}
\dim\Lambda^nX&=d(d-1)\cdots (d-n+1)/n!\neq0,\\
\dim\Lambda^{n+1}X&=d(d-1)\cdots(d-n)/(n+1)!=0
\end{align*}
where $d=\dim X$. This is only possible if $d=n$, meaning $\dim\Lambda^nX=n!/n!=1$ giving the result by Lemma~\ref{LemInvDim}. 
\end{proof}

\section{Decomposition of $\Ver_p(G)$}\label{SecVerpG}

Let $G$ be a simple connected algebraic group over an algebraically closed field $\bk$ with characteristic $p>0$. Throughout this section we assume that $p>h$, where $h$ is the Coxeter number of $G$ as defined in \ref{SecGroups}.

\subsection{Notation for algebraic groups}\label{SecGroups}

\subsubsection{}Fix a maximal torus $T$ of $G$, giving a weight lattice $X(T)$. Denote by $\Phi$ the set of roots and $\mZ\Phi\subseteq X(T)$ the root lattice. Let $r$ be the rank of $G$, and fix a basis of simple roots $\alpha_1,\dots,\alpha_r$ with corresponding set of positive roots $\Phi^+$. Denote by $X(T)^+$ the set of dominant weights with respect to $\Phi^+$, and $\varpi_1,\dots,\varpi_r\in X(T)^+$ the fundamental weights. We write $\rho$ for the sum of the fundamental weights, or equivalently half the sum of the positive roots. Let $(\cdot,\cdot)$ be the bilinear form on $X(T)$ coming from the Killing form, and denote $\langle\lambda,\alpha^\vee\rangle=2(\lambda,\alpha)/(\alpha,\alpha)\in\mZ$ for $\lambda\in X(T)$ and $\alpha\in\Phi$, so that $\langle\varpi_i,\alpha_j^\vee\rangle=\delta_{ij}$. We write $\theta_l$ for the highest (long) root and $\theta_s$ for the highest short root. The Coxeter number is $h=1+\sum_{i=1}^r k_i$ where the integers $k_i$ are such that $\theta_l=\sum_{i=1}^r k_i\alpha_i$. We write $\Rep G$ for the category of finite-dimensional algebraic $G$-representations, and $L(\lambda)$ for the simple representation with highest weight $\lambda\in X(T)^+$.

\subsubsection{}\label{SecWeightsA} Denote by $W$ the Weyl group consisting of automorphisms $X(T)\to X(T)$ generated by reflections $s_\alpha$ for $\alpha\in\Phi$. We define the ($p$-scaled) affine Weyl group $\widetilde W_p\cong\mZ\Phi\rtimes W$ to be the group of automorphisms of $X(T)$ generated by the shifted reflections $s_{\alpha,kp}(\lambda)\coloneqq s_\alpha(\lambda)+kp\alpha$ for $k\in\mZ$, and we define the extended affine Weyl group $\widehat W_p\cong X(T)\rtimes W$ to be the group generated by $s_{\alpha,p\mu}(\lambda)\coloneqq s_\alpha(\lambda)+p\mu$ for $\mu\in X(T)$ (see e.g. \cite{Ga}, \cite[IV \S 2]{Bo}). The ``dot action'' of $w$ in $W$, $\widetilde W_p$ or $\widehat W_p$ is defined as $w\cdot\lambda=w(\lambda+\rho)-\rho$. Let
$$A=\{\lambda\in X(T)\mid0<\langle\lambda+\rho,\alpha^\vee\rangle<p,\ \forall\alpha\in\Phi^+\}$$ be the set of weights in the open fundamental alcove of the dot action of $\widetilde W_p$. The fundamental alcove is a simplex with the far wall given by the hyperplane $\langle\lambda+\rho,\theta_s^\vee\rangle=p$ where $\theta_s$ is highest short root, and the other walls restricting $\lambda\in A$ to $X(T)^+$, see \cite[V \S 3]{Bo}. We recall that $\theta_s^\vee$ is the highest (long) root of the dual root system $\Phi^\vee$ (see e.g. \cite[Ex 13.6]{Hu}), and since $\Phi^\vee$ has the same Coxeter number as $\Phi$, we have $h=1+\sum_i k_i'$ where $\theta_s^\vee=\sum_{i=1}^rk_i'\alpha_i^\vee$. By linearity of $\langle\cdot,\cdot\rangle$, we have
$$\langle\rho,\theta_s^\vee\rangle=\sum_{i=1}^rk_i'\langle\rho,\alpha_i^\vee\rangle=\sum_{i=1}^rk_i'=h-1$$
and hence we can rewrite $A$ as
$$A=\{\lambda\in X(T)^+\mid\langle\lambda,\theta_s^\vee\rangle\leq p-h\}.$$

\subsubsection{}\label{SecZG} By the description of the centre $Z(G)$ of $G$ as the common kernel in $T$ of all coroots, see \cite[II.1.6]{Jantzen}, it follows that $Z(G)$ is the diagonalisable group scheme corresponding to the abelian group $H\coloneqq X(T)/\mZ\Phi$ (so $\bk[Z(G)]$ is isomorphic as a Hopf algebra to the group algebra $\bk H$). In particular the representation category of $Z(G)$ is the category of $X(T)/\mZ\Phi$-graded vector spaces, $\Vecc_{X(T)/\mZ\Phi}$. We also recall that $X(T)/\mZ\Phi$ is in bijection with the set of minuscule weights of $G$, which are weights $\varpi\in X(T)^+$ satisfying any of the following equivalent conditions as described in \cite[VIII \S 7.3]{Bo}:
\begin{enumerate}\samepage
\item All weights of $L(\varpi)$ are in a single $W$-orbit;
\item $\varpi$ satisfies $\langle\varpi,\alpha^\vee\rangle\leq1$ for all $\alpha\in\Phi^+$;
\item $\varpi$ is either zero or satisfies $\langle\varpi,\theta_s^\vee\rangle=1$.
\end{enumerate}
Every non-zero minuscule weight is a fundamental weight, and if $\varpi$ is minuscule then all weight spaces of $L(\varpi)$ are 1-dimensional.

\subsubsection{Table of simple groups}\label{TabGroups} Each Cartan type has a unique simply-connected group (i.e. the simple group with maximal centre) and adjoint-type group (i.e. the simple group with trivial centre), and we list these along with some labelling conventions below.

\vspace{0.5em}{\noindent
\begin{tabular}{r|l|l|l|l|l|l}
Type & Rank & Diagram & $h$ & Simply-con. & Adjoint & Centre of simply-con. \\\hline
$A_r$ & $r\geq1$ & \dynkin[labels={1,2,,r}]A{} & $r+1$ & $\SL_{r+1}$ & $\PGL_{r+1}$ & $\mZ/(r+1)$ \\\hline
$B_r$ & $r\geq2$ & \dynkin[labels={1,2,,r\scalebox{0.6}[1.0]{$-$}1,r}]B{} & $2r$ & $\Spin_{2r+1}$ & $\SO_{2r+1}$ & $\mZ/2$\\\hline
$C_r$ & $r\geq2$ & \dynkin[labels={1,2,,r\scalebox{0.6}[1.0]{$-$}1,r}]C{} & $2r$ & $\Sp_{2r}$ & $\Sp_{2r}^\ad$ & $\mZ/2$\\\hline
$D_r$ & $r\geq4$ & \dynkin[labels={1,2,,r\scalebox{0.6}[1.0]{$-$}2,r\scalebox{0.6}[1.0]{$-$}1,r},label directions={,,,right,,}]D{} & $2r-2$ & $\Spin_{2r}$ & $\PSO_{2r}$ & $\begin{cases}\mZ/2\times\mZ/2 &r\text{ even}\\\mZ/4 &r\text{ odd}\end{cases}$ \\\hline
$E_6$ & 6 & \dynkin[labels={5,6,4,3,2,1}]E6 & 12 & $E_6$ & $E_6^\ad$ & $\mZ/3$ \\\hline
$E_7$ & 7 & \dynkin[labels={6,7,5,4,3,2,1}]E7 & 18 & $E_7$ & $E_7^\ad$ & $\mZ/2$ \\\hline
$E_8$ & 8 & \dynkin[labels={7,8,6,5,4,3,2,1}]E8 & 30 & $E_8$ & $E_8$ & $1$ \\\hline
$F_4$ & 4 & \dynkin[labels={4,3,2,1}]F4 & 12 & $F_4$ & $F_4$ & $1$ \\\hline
$G_2$ & 2 & \dynkin[labels={2,1}]G2 & 6 & $G_2$ & $G_2$ & $1$ \\\hline
\end{tabular}}\vspace{0.5em}

\subsubsection{} Denote by $\Ver_p(G)$ the semisimplification of the category $\Tilt G$ of tilting modules of~$G$. We write $\Ver_p\coloneqq\Ver_p(\SL_2)$ and $\Ver_p^+\coloneqq\Ver_p(\PGL_2)$, and abbreviate objects $L(k\varpi_1)$ as $L_k$ in $\Tilt\SL_2$ and $\Ver_p$ (beware that in some papers $L(k\varpi_1)$ is written as $L_{k+1}$ instead, e.g. \cite{CEO2, Os}). The tensor categories $\Ver_p(G)$ were introduced in \cite{GK} (under the name $\mathbf{Sm}$, referring to ``small'' objects in $\Rep G$) for simply-connected groups via an alternate definition using a principal $\SL_2$. The connection to tilting modules was later observed in \cite{GM}. We refer to \cite[Appendix~E]{Jantzen} for more information on tilting modules.

\subsection{Images of representations}\label{SecImages} For $G$ simply-connected, we recall from \cite{GK} that a principal $\SL_2$ is a group morphism $\SL_2\to G$ sending non-trivial unipotent elements to principal unipotent elements. This is unique up to conjugacy, and one such choice is given on the level of Lie algebras by sending the single coroot of $\mathfrak{sl}_2$ to the sum of the positive coroots of ${\rm Lie}(G)$. The following proposition follows from the description of $\Ver_p(G)$ in \cite{GK} and its relation to a principal $\SL_2$.

\begin{prop}\label{PropVerBasics}
\begin{enumerate}
\item The simple objects of $\Ver_p(G)$ are precisely the images of the objects $L(\lambda)\in\Tilt G$ for $\lambda\in A$, with the set $A$ defined in Section~\ref{SecWeightsA}.
\item For a central quotient $G'=G/H$ (so $H\subseteq Z(G)$ is necessarily finite), the inclusion functor $\Rep(G')\to\Rep(G)$ induces a commutative diagram of symmetric monoidal functors:
$$\begin{tikzcd}
\Rep(G') \arrow{d} &[3em] \Tilt(G') \arrow{d} \arrow[swap]{l}{\rm inclusion} \arrow{r}{\rm semisimplification} &[4em] \Ver_p(G') \arrow{d} \\
\Rep(G) & \Tilt(G) \arrow[swap]{l}{\rm inclusion} \arrow{r}{\rm semisimplification} & \Ver_p(G).
\end{tikzcd}$$
\item For a principal $\SL_2\to G$, the restriction functor $\Rep(G)\to\Rep(\SL_2)$ induces a commutative diagram of symmetric monoidal functors:
$$\begin{tikzcd}
\Rep(G) \arrow{d} &[3em] \Tilt(G) \arrow{d} \arrow[swap]{l}{\rm inclusion} \arrow{r}{\rm semisimplification} &[4em] \Ver_p(G) \arrow{d} \\
\Rep(\SL_2) & \Tilt(\SL_2) \arrow[swap]{l}{\rm inclusion} \arrow{r}{\rm semisimplification} & \Ver_p.
\end{tikzcd}$$
The functor $\Rep(G)\to\Rep(\SL_2)$ corresponds to the map
$$X(T)\to\mZ,\quad\lambda\mapsto\sum_{\alpha\in\Phi^+}\langle\lambda,\alpha^\vee\rangle$$
on the weight spaces of $G$ and $\SL_2$, where we identify the latter space with $\mZ$.
\end{enumerate}
\end{prop}

The Lie algebra ${\rm Lie}(G)$ equipped with the adjoint action is a simple $G$-module with highest weight given by the highest (long) root $\theta_l$, and hence is isomorphic to $L(\theta_l)$ as a $G$-module.

\begin{lemma}\label{LemAdjointTilting}
$L(\theta_l)$ is a tilting module, and the image of $L(\theta_l)$ in $\Ver_p(G)$ is zero if and only if $p=h+1$. Moreover, the image of $L(\theta_l)$ in $\Ver_p$ under the functor in \ref{PropVerBasics}(3) lies in $\Ver_p^+$.
\end{lemma}

\begin{proof}
If the type of $G$ is not $G_2$ then we have $\langle\theta_l,\theta_s^\vee\rangle=2$. This means that, for $p\geq h+2$, we have $\theta_l\in A$ and hence $L(\theta_l)$ is a simple tilting with a non-zero image in $\Ver_p(G)$. For $p=h+1$, $\theta_l$ is on the far wall of the fundamental alcove, so it is a negligible tilting by the theory of \cite[\S E]{Jantzen}. 

If $G=G_2$, then $\langle\theta_l,\theta_s^\vee\rangle=3$, so $\theta_l\in A$ unless $p<h+3=9$. Hence the only special case with $p>h$ is $p=7=h+1$. In this case $\theta_l$ lies on the wall of the next alcove up, so $L(\theta_l)$ is again a (negligible) tilting module.

One can show that $L(\theta_l)$ lands in $\Ver_p^+$ by proving that roots are sent to even integers under the map in $\ref{PropVerBasics}$, or alternately by directly computing the image of $L(\theta_l)$ in each type as we do in Table~\ref{TabImages} below.
\end{proof}

\subsubsection{}\label{SecWeylFactors}If $p>h$ and $\lambda\in A$, then the restriction of $L(\lambda)$ to the principal $\SL_2$ must be a tilting module. The $\SL_2$ Weyl module $\Delta_m$ with highest weight $m$ has weights $m,m-2,\dots,-m$, and we write $T_m$ for the unique $\SL_2$ tilting module with highest weight $m$. We recall from \cite{TW} that if $T_m$ is a negligible tilting module (that is $m\geq p-1$), then the Weyl composition factors of $T_m$ are either of the form $\Delta_{ap-1}$ with $a\in\mZ_{>0}$, or occur in pairs of the form $\Delta_{ap+b-1},\Delta_{ap-b-1}$ with $a\in\mZ_{>0}$ and $1\leq b\leq p-1$. Using this information one can determine the image of a simple $L(\lambda)\in\Ver_p(G)$ in $\Ver_p$ under the functor in Proposition~\ref{PropVerBasics}(3), by collecting the weights into sequences of the form $m,m-2,\cdots,-m$ to determine the Weyl modules and then discarding those which combine to give negligible tiltings. Lemma~\ref{LemCharacters} gives a neat algebraic realisation of this process.

In Table~\ref{TabImages} we list the Weyl composition factors of some important representations in each type, and then one can deduce the image of a listed representation in $\Ver_p$ for $p>h$ by the above procedure. We give some examples that will be important for the main theorem:
\begin{enumerate}
\item In type $A_r$, $L(\theta_l)$ has image $L_2\oplus L_4\oplus\dots\oplus L_{2s}$ in $\Ver_p$ where $s=\min(r,p-2-r)$.
\item In types $B_r$ and $C_r$, $L(\theta_l)$ has image $L_2\oplus L_6\oplus L_{10}\oplus\dots\oplus L_{4s-2}$ in $\Ver_p$ where $s=\min(r,\frac{p-1}2-r)$.
\item In type $D_r$, $L(\theta_l)$ has image $L_2\oplus L_{2r-2}$ in $\Ver_{2r+1}$ whenever $2r+1$ is prime.
\item In type $E_7$, $L(\theta_l)$ has image $L_2\oplus L_{14}$ and $L(\varpi_1)$ has image $L_9$ in $\Ver_{23}$.
\item In type $G_2$, $L(\theta_l)$ has image $L_2\oplus L_{10}$ in $\Ver_p$ for $p\geq13$. 
\end{enumerate}

The images of the half-spin representations in type $D_r$ will also be important for the main theorem, and since their derivation is more involved we give the full details in Proposition~\ref{PropTypeD}. That the type $D_r$ adjoint and half-spin representations undergo so much cancellation in characteristic $2r+1$ may appear surprising, so we give a more conceptual explanation for this case in Remark~\ref{RemOSp}.

\subsubsection{Table of small and adjoint representations}\label{TabImages} Below we list the smallest non-trivial representation and adjoint representation in each type, and the Weyl composition factors of their images over the principal $\SL_2$. The classical cases are well known, and the exceptional cases are directly computed in \cite{St}. Note that for type $B_2=C_2$, we use the $C_2$ labelling convention.

\vspace{0.5em}{\noindent
\begin{tabular}{r|l|l}
Type & Smallest non-trivial & Adjoint representation\\\hline
$A_r$ & $L(\varpi_1)\mapsto\Delta_r$ & $L(\varpi_1+\varpi_r)\mapsto[\Delta_2,\Delta_4,\Delta_6,\dots,\Delta_{2r}]$\\\hline
$B_r$ & $L(\varpi_1)\mapsto\Delta_{2r}$ & $L(\varpi_2)\mapsto[\Delta_2,\Delta_6,\Delta_{10},\dots,\Delta_{4r-2}]$\\\hline
$C_r$ & $L(\varpi_1)\mapsto\Delta_{2r-1}$ & $L(2\varpi_1)\mapsto[\Delta_2,\Delta_6,\Delta_{10},\dots,\Delta_{4r-2}]$\\\hline
$D_r$ & $L(\varpi_1)\mapsto[\Delta_0,\Delta_{2r-2}]$ & $L(\varpi_2)\mapsto[\Delta_2,\Delta_6,\Delta_{10},\dots,\Delta_{4r-6},\Delta_{2r-2}]$\\\hline
$E_6$ & $L(\varpi_1)\mapsto[\Delta_0,\Delta_8,\Delta_{16}]$ & $L(\varpi_6)\mapsto[\Delta_2,\Delta_8,\Delta_{10},\Delta_{14},\Delta_{16},\Delta_{22}]$\\\hline
$E_7$ & $L(\varpi_1)\mapsto[\Delta_9,\Delta_{17},\Delta_{27}]$ & $L(\varpi_6)\mapsto[\Delta_2,\Delta_{10},\Delta_{14},\Delta_{18},\Delta_{22},\Delta_{26},\Delta_{34}]$\\\hline
$E_8$ & same as adjoint & $L(\varpi_1)\mapsto[\Delta_2,\Delta_{14},\Delta_{22},\Delta_{26},\Delta_{34},\Delta_{38},\Delta_{46},\Delta_{58}]$\\\hline
$F_4$ & $L(\varpi_1)\mapsto[\Delta_8,\Delta_{16}]$ & $L(\varpi_4)\mapsto[\Delta_2,\Delta_{10},\Delta_{14},\Delta_{22}]$\\\hline
$G_2$ & $L(\varpi_1)\mapsto\Delta_6$ & $L(\varpi_2)\mapsto[\Delta_2,\Delta_{10}]$\\\hline
\end{tabular}}\vspace{0.5em}

\begin{lemma}\label{LemCharacters}
Let $T,T'$ be tilting modules of $\SL_2$ and consider the character polynomial
$$\ch_T(x)=\sum_{\lambda\in\mZ}m_T(\lambda)x^\lambda,$$
writing $m_T(\lambda)$ for the multiplicity of the weight $\lambda$ in $T$, and similarly for $T'$. Let $\omega,\xi\in\mC$ be primitive $p$-th and $2p$-th roots of unity respectively. Then the images of $T$ and $T'$ in $\Ver_p$ are isomorphic if and only if $\ch_T(\omega)=\ch_{T'}(\omega)$ and $\ch_T(\xi)=\ch_{T'}(\xi)$. These two equalities are equivalent if the weights of $T$ and $T'$ together are either all even or all odd.
\end{lemma}

\begin{proof}
Writing $[T:\Delta_\lambda]$ for the multiplicity of $\Delta_\lambda$ in the Weyl filtration of $T$, we have
$$\ch_T(x)=\sum_{\lambda\geq0}[T:\Delta_\lambda](x^\lambda+x^{\lambda-2}+\cdots+x^{-\lambda})=\sum_{\lambda\geq0}[T:\Delta_\lambda]\frac{x^{\lambda+1}-x^{-\lambda-1}}{x-x^{-1}}.$$
We have $\omega^{ap}-\omega^{-ap}=0$ and $(\omega^{ap+b}-\omega^{-ap-b})+(\omega^{ap-b}-\omega^{-ap+b})=0$ for all $a,b\in\mZ$, and similarly $\xi^{ap}-\xi^{-ap}=0$ and $(\xi^{ap+b}-\xi^{-ap-b})+(\xi^{ap-b}-\xi^{-ap+b})=0$. Thus, following Section~\ref{SecWeylFactors}, Weyl factors that combine to give negligible summands of $T$ are annihilated when contributing to $\ch_T(\omega),\ch_T(\xi)$. So, writing $\overline T$ for the image of $T$ in $\Ver_p$, we have
$$\ch_T(\omega)=\sum_{\lambda=0}^{p-2}[\overline T:L_\lambda]\frac{\omega^{\lambda+1}-\omega^{-\lambda-1}}{\omega-\omega^{-1}},\quad \ch_T(\xi)=\sum_{\lambda=0}^{p-2}[\overline T:L_\lambda]\frac{\xi^{\lambda+1}-\xi^{-\lambda-1}}{\xi-\xi^{-1}}.$$
Thus $\overline T\cong\overline T'$ implies $\ch_T(\omega)=\ch_{T'}(\omega)$ and $\ch_T(\xi)=\ch_{T'}(\xi)$, and now we show the converse. We have $\frac{\omega-\omega^{-1}}{\omega}(\ch_T(\omega)-\ch_{T'}(\omega))=f_0(\omega)$ and $\frac{\xi-\xi^{-1}}{\xi}(\ch_T(\xi)-\ch_{T'}(\xi))=f_1(\xi)$ where
$$f_i(x)=\sum_{\lambda=0}^{p-2}([\overline T:L_\lambda]-[\overline T':L_\lambda])(x^\lambda-(-1)^ix^{p-\lambda-2})\in\mQ[x].$$
The minimal polynomials of $\omega,\xi$ over $\mQ$ have degree $p-1$, but $f_i(x)$ has degree at most $p-2$. So if $\ch_T(\omega)=\ch_{T'}(\omega)$ then $f_0(x)=0$, and if $\ch_T(\xi)=\ch_{T'}(\xi)$ then $f_1(x)=0$. When both of these hold, the coefficients of $x^\lambda$ in $f_0(x)+f_1(x)=0$ show that $\overline T,\overline T'$ have the same composition multiplicities, so $\overline T\cong\overline T'$. If all weights of $T,T'$ have the same parity, it is enough to have only one of these, since the only linear relations between the terms $y_\lambda=x^\lambda-(-1)^ix^{p-\lambda-2}$ are between $y_\lambda$ and $y_{p-\lambda-2}$ whose indices have opposite parity.
\end{proof}

\begin{prop}\label{PropTypeD}
Suppose $p=2r+1\geq7$. The half-spin representations $L(\varpi_r)$ and $L(\varpi_{r-1})$ in type $D_r$ both have image
$$\begin{cases}
L_{r-1} & \text{ if }r\equiv1\text{ or }r\equiv2\mod4\\
L_r & \text{ if }r\equiv3\text{ or }r\equiv0\mod4
\end{cases}$$
under the functor $\Ver_p(\Spin_{2r})\to\Ver_p$.
\end{prop}

\begin{proof}
We recall the following well-known realisation of the weight space of type $D_r$. Let
\begin{align*}
e_1&=\varpi_1, & e_i&=\varpi_i-\varpi_{i-1}\text{ for }2\leq i\leq r-2,\\
e_{r-1}&=\varpi_r+\varpi_{r-1}-\varpi_{r-2}, & e_r&=\varpi_r-\varpi_{r-1}.
\end{align*}
Then we have $(e_i,e_j)=\delta_{ij}$ for all $1\leq i,j\leq r$. The roots are given by $\pm e_i\pm e_j$ with $i\neq j$, the simple roots are $\alpha_i=e_i-e_{i+1}$ for $1\leq i\leq r-1$ and $\alpha_r=e_{r-1}+e_r$, and the positive roots are $e_i+e_j$ and $e_i-e_j$ for all $j>i$. This means the restriction to the principal $\SL_2$ in Proposition~\ref{PropVerBasics} (3) is given by a map $X(T)\to\mZ$ sending $e_i$ to $2r-2i$.

We have $\varpi_r=\frac12(e_1+\cdots+e_r)$ and $\varpi_{r-1}=\frac12(e_1+\cdots+e_{r-1}-e_r)$, hence the weights of $L(\varpi_r)$ (respectively $L(\varpi_{r-1})$) are precisely all weights of the form $\frac12(\pm e_1\pm\cdots\pm e_r)$ with an even (respectively odd) number of minus signs. Under the map $X(T)\to\mZ$ these become $\pm1\pm2\pm\cdots\pm(r-1)$ for all possible combinations of $+$ and $-$ (with no even-odd split since $e_r\mapsto0$). Hence all the weights have the same parity as the highest weight $\frac{r(r-1)}2$, and
$$\ch_X(x)=\prod_{i=1}^{r-1}(x^i+x^{-i})=\prod_{i=1}^{r-1}\frac{x^{2i}-x^{-2i}}{x^i-x^{-i}}$$
where $X$ is the image of either $L(\varpi_r)$ or $L(\varpi_{r-1})$ in $\Tilt\SL_2$. Substituting $x$ for a primitive $p$-th root of unity $\omega$, we have
\begin{align*}
\ch_X(\omega)&=\frac{\left(\prod_{2\leq j\leq r,\ j\text{ even}}(\omega^j-\omega^{-j})\right)\left(\prod_{r+1\leq j\leq p-3,\ j\text{ even}}(-\omega^{p-j}+\omega^{-p+j})\right)}{\prod_{i=1}^{r-1}(\omega^i-\omega^{-i})}\\
&=\frac{\left(\prod_{2\leq j\leq r,\ j\text{ even}}(\omega^j-\omega^{-j})\right)\left(\prod_{3\leq k\leq r,\ k\text{ odd}}(-\omega^k+\omega^{-k})\right)}{\prod_{i=1}^{r-1}(\omega^i-\omega^{-i})}\\
&=(-1)^l\frac{\prod_{j=2}^r(\omega^j-\omega^{-j})}{\prod_{i=1}^{r-1}(\omega^i-\omega^{-i})}=(-1)^l\frac{\omega^r-\omega^{-r}}{\omega-\omega^{-1}}=(-1)^l\frac{-\omega^{r+1}+\omega^{-r-1}}{\omega-\omega^{-1}}
\end{align*}
where $l=\lfloor\frac{r-1}2\rfloor$. Now $\ch_X(\omega)=\ch_Y(\omega)$ where $Y=L_{r-1}$ if $l$ is even and $L_r$ if $l$ is odd, and the weights of $X,Y$ have the same parity. Thus, the result follows from Lemma~\ref{LemCharacters}.
\end{proof}

\subsection{Decomposition of Verlinde categories}\label{SecVerDecomp}

\subsubsection{}\label{SecVerSetup} Let $\Ver_p^0(G)$ be the subcategory of invertible objects in $\Ver_p(G)$, and $\Ver_p^+(G)$ the subcategory spanned by $L(\lambda)$ with $\lambda\in\mZ\Phi\cap A$, that is $\Ver_p^+(G)\simeq\Ver_p(G/Z(G))$. Write $Z=Z(G)$ for the centre of $G$. Fix a principal $\SL_2\to G$, and let $z:\mZ/2\to Z$ be the restriction of this map to the centres (note that $Z(\SL_2)\cong\mZ/2$ since $p\geq h+1\geq3$). Such a restriction exists because $Z$ is the kernel of the adjoint map $G\to\GL({\rm Lie}(G))$ and the action of $Z(\SL_2)$ on the Lie algebra ${\rm Lie}(G)$ is trivial, since the restriction of ${\rm Lie}(G)$ to $\SL_2$ is given by the adjoint representation in Table~\ref{TabImages} which has only even weights.

By Lemma~\ref{LemAdjointTilting}, ${\rm Lie}(G)$ has an image in $\Ver_p^+$ which we denote by $\fg$. The Lie bracket is a homomorphism of $G$-representations, so its image in $\Ver_p^+$ endows $\fg$ with a Lie algebra structure. Denote by $\phi:L_2\to\fg$ the image in $\Ver_p^+$ of the inclusion of the principal $\mathfrak{sl}_2$ into ${\rm Lie}(G)$. Since this inclusion is a Lie algebra homomorphism in $\Rep\SL_2$, $\phi$ is a Lie algebra homomorphism with the Lie algebra structure of $L_2$ coming from the fundamental group of $\Ver_p^+$.

\begin{theorem}\label{ThmVerDecomp}
If $G$ is a simple algebraic group with Coxeter number $h<p$, then we have the following:
\begin{enumerate}
\item $\Ver_p^0(G)\simeq\Rep_\sVec(Z,z)$;
\item $\Ver_p^+(G)\simeq\Rep_{\Ver_p^+}(\fg,\phi)$;
\item The only invertible object in $\Ver_p^+(G)$ is $\unit$, and the only tensor subcategories of $\Ver_p^+(G)$ are $\Vecc$ and itself;
\item $\Ver_p(G)\simeq\Ver_p^+(G)\boxtimes\Ver_p^0(G)\simeq\Rep(\fg,\phi)\boxtimes\Rep(Z,z)$.
\end{enumerate}
\end{theorem}

\begin{example}\label{ExSLn}
For $G=\SL_n$ we can construct the Lie algebra $\fg$ exactly, as shown in \cite{Ve2}. For an object $X\in\Ver_p$, we define the Lie algebra $\mathfrak{gl}(X)=X\otimes X^*$ with bracket
$$[\cdot,\cdot]=(\id_X\otimes\ev_X\otimes\id_{X^*})\circ(\id_{\mathfrak{gl}(X)\otimes\mathfrak{gl}(X)}-c_{\mathfrak{gl}(X),\mathfrak{gl}(X)}).$$
We define $\mathfrak{sl}(X)$ as the Lie subalgebra given by the kernel of $\ev_{X^*}:\mathfrak{gl}(X)\to\unit$, so in particular for $2\leq n\leq p-2$ the Lie algebra $\mathfrak{sl}(L_{n-1})$ has underlying object
$$\mathfrak{sl}(L_{n-1})=L_2\oplus L_4\oplus L_6\oplus \cdots\oplus L_{2s-2}\quad\mbox{ with }\;s=\begin{cases}
n&\mbox{if $n<p/2$,}\\
p-n&\mbox{otherwise.}
\end{cases}
$$
The equivalence in Theorem~\ref{ThmVerDecomp} for $\SL_n$ is
$$\Ver_p(\SL_n)\simeq\Rep(\mathfrak{sl}(L_{n-1}),\phi)\boxtimes\Rep(\mu_n,z)$$
where $\phi$ is the inclusion $L_2\hookrightarrow\mathfrak{sl}(L_{n-1})$. If $n$ is odd then the invertibles in $\Ver_p(\SL_n)$ are even by the characterisation of invertibles below, hence $\Rep(\mu_n,z)\simeq\Vecc_{\mZ/n}$. If $n$ is even then we instead have $\Rep(\mu_n,z)\simeq\sVec_{\mZ/n,z}$ with $z$ assigning the generator of $\mZ/n$ an odd braiding. Since $L_{n-1}\otimes\overline{\unit}\cong L_{p-n-1}$, we have a Lie algebra isomorphism $\mathfrak{sl}(L_{n-1})\cong\mathfrak{sl}(L_{p-n-1})$ giving the level-rank duality $\Ver_p^+(\SL_n)\simeq\Ver_p^+(\SL_{p-n})$.
\end{example}

The decomposition $\Ver_p(G)\simeq\Ver_p^+(G)\boxtimes\Ver_p^0(G)$ is remarked on in \cite[\S 4.1]{CEO2} and \cite[Remark~6.2]{Gr}, although a detailed proof of this (along with the connection to $Z$ and $\fg$) appears to be absent from the literature. We give a full proof of Theorem~\ref{ThmVerDecomp} below. 

{The following lemmas are results of \cite[II.E.12]{Jantzen} and \cite[Lemma~1.1]{Gr}.} Recall that $\widetilde{W}_p$ is the affine Weyl group and $\widehat{W}_p$ the extended affine Weyl group.

\begin{lemma}\label{LemTensorDecomp}
For $\lambda,\mu,\nu\in A$, the multiplicity of $L(\nu)$ in $L(\lambda)\otimes L(\mu)$ in $\Ver_p(G)$ is
$$M_{\lambda,\mu}^\nu=\sum_{w\in\widetilde W_p}\varepsilon(w)m_\lambda(w\cdot\nu-\mu)$$
where $\varepsilon:\widetilde W_p\to\{1,-1\}$ is the sign homomorphism with $\varepsilon(s_{\alpha,k})=-1$ for all $\alpha\in\Phi$ and $k\in\mZ$, and $m_\lambda(x)$ is the multiplicity of the weight $x$ in $L(\lambda)$ in $\Rep G$.
\end{lemma}

\begin{lemma}\label{LemAlcove}
Suppose $\sigma\in\widehat W_p$ satisfies $\sigma\cdot A=A$. Then $L(\sigma\cdot0)$ is invertible, and for all $\lambda\in A$ we have
$$L(\sigma\cdot\lambda)\cong L(\sigma\cdot0)\otimes L(\lambda).$$
\end{lemma}

So for any $\sigma\in\widehat W_p$ sending $A$ to itself, there is an invertible $L(\sigma\cdot0)$ whose action on $A$ via the tensor product is the same as $\sigma$. Now we show that this produces all invertible objects, and moreover the invertibles are in one-to-one correspondence with minuscule weights.

\begin{prop}\label{PropInvertibles}
The following are equivalent for $\mu\in A$:
\begin{enumerate}
\item $\mu=(p-h)\varpi$ for some minuscule weight $\varpi$.
\item $L(\mu)$ is invertible in $\Ver_p(G)$.
\item $\Hom(L(\theta_l),L(\mu)\otimes L(\mu)^*)=0$ in $\Ver_p(G)$.
\end{enumerate}
Moreover, $L(\mu)\mapsto\mu+\mZ\Phi$ is a bijection between invertibles and cosets of the root lattice $\mZ\Phi$ in $X(T)$.
\end{prop}

\begin{proof}
We may assume $G$ is simply connected, since $L(\lambda)$ is invertible in $\Ver_p(G')$ for $G'=G/H$ and $H\subseteq Z(G)$ if and only if it is invertible in $\Ver_p(G)$. For each non-zero minuscule weight $\varpi$, an element $\sigma\in\widehat W_p$ with $\sigma\cdot A=A$ and $\sigma\cdot0=(p-h)\varpi$ is given in Table~\ref{TabMinuscule}. This means $L((p-h)\varpi)$ is invertible by Lemma~\ref{LemAlcove}, so (1) implies (2). That (2) implies (3) is clear from the fact that $L(\mu)\otimes L(\mu)^*\cong\unit$ for $L(\mu)$ invertible.

We now show that (3) implies (1). If $L(\theta_l)$ is zero then by Lemma~\ref{LemAdjointTilting} we have $p=h+1$, but then $\langle\lambda,\theta_s^\vee\rangle\in\{0,1\}$ for every $\lambda\in A$, so every simple in $\Ver_p(G)$ is of the form $L((p-h)\varpi)=L(\varpi)$ for minuscule $\varpi$. So we may assume $p>h+1$ and $L(\theta_l)\neq0$. Our method is identical to \cite[Lemma 3]{Sa} which gives a similar result for quantum groups. By adjunction, $\Hom(L(\theta_l),L(\mu)\otimes L(\mu)^*)=0$ if and only if $L(\mu)$ is not a summand of $L(\theta_l)\otimes L(\mu)$, that is $M_{\theta_l,\mu}^\mu=0$. The weights of $L(\theta_l)$ are the roots (each with multiplicity 1) and zero (with multiplicity the rank $r$ of $G$), so by the formula in Lemma~\ref{LemTensorDecomp} we have
$$M_{\theta_l,\mu}^\mu\;=\; r-\# S_{\rm odd}+\# S_{\rm even}$$
where $S_{\rm odd}\subset\Phi$ (resp. $S_{\rm even}$) is the set of roots $\alpha$ such that $\mu+\alpha=w\cdot\mu$ for some $w\in\widetilde W_p$ equal to an odd (resp. even) product of reflections. Then $M_{\theta_l,\mu}^\mu=0$ implies $\# S_{\rm odd}\geq r$.

We claim that all roots $\alpha\in S_{\rm odd}$ have $\mu+\alpha=s\cdot\mu$ for some reflection $s$ through a wall of the fundamental alcove. To prove this, suppose $\alpha\in S_{\rm odd}\cup S_{\rm even}$ and consider a sequence $\mu=\mu_0,\mu_1,\mu_2,\dots,\mu_\ell=\mu+\alpha$ in $X(T)$ where each step from $\mu_i$ to $\mu_{i+1}$ is a reflection in $\widetilde W_p$ through a wall of the alcove containing $\mu_i$. We have $\mu_{i+1}-\mu\in\mZ\Phi$ for all $i$, since
$$\mu_{i+1}\;=\;s_\beta\cdot\mu_i+\gamma\;=\;\mu_i-\langle\mu_i+\rho,\beta^\vee\rangle\beta+\gamma\;\in\;\mu_i+\mZ\Phi$$
for some $\beta\in\Phi$ and $\gamma\in\mZ\Phi$. However, if we assume that this sequence is minimal, then the value $\|\mu_i-\mu\|$ strictly increases with $i$, where $\|\lambda\|=\sqrt{(\lambda,\lambda)}$ is the geometric length of a weight $\lambda$. This means that $\|\mu_1-\mu\|$ is at least the length of a short root, $\|\mu_2-\mu\|$ is at least the length of a long root, and if $\ell\geq3$ then $\|\mu_\ell-\mu\|=\|\alpha\|$ is greater than the length of any root in $\Phi$, a contradiction. Thus all $\alpha\in S_{\rm odd}$ have $\ell=1$, proving the claim.

The reflections $s_{\alpha_i}$ for each simple root $\alpha_i$ have $\mu-\alpha_i=s_{\alpha_i}\cdot\mu$ if and only if $\langle\mu,\alpha_i^\vee\rangle=0$. The only other reflection through a wall of the fundamental alcove is through the far wall, which sends $\mu$ to $\mu+\theta_s$ if and only if $\langle\mu,\theta_s^\vee\rangle=p-h$. In order to have $\# S_{\rm odd}\geq r$, at least $r$ of these conditions on $\mu$ must be satisfied. If $\langle\mu,\alpha_i^\vee\rangle=0$ for all $r$ simple roots $\alpha_i$, then $\mu=0$ and we are done. Otherwise, we must have $\langle\mu,\theta_s^\vee\rangle=p-h$ and $\langle\mu,\alpha_i^\vee\rangle=0$ for $r-1$ simple roots, meaning $\mu=k\varpi$ for some fundamental weight $\varpi$ and $k=(p-h)/\langle\varpi,\theta_s^\vee\rangle$. But $p$ is prime, and it can be checked that $h$ is a multiple of some prime factor of $\langle\varpi,\theta_s^\vee\rangle$ for any fundamental $\varpi$, so we require $\langle\varpi,\theta_s^\vee\rangle=1$ to have $k\in\mZ$. Thus $\varpi$ is minuscule, completing the proof that (3) implies (1). Finally, one can check case by case that the bijection $L(\mu)\mapsto\mu+\mZ\Phi$ in the lemma holds using Table~\ref{TabMinuscule}.
\end{proof}

\begin{remark}
A similar classification of invertibles is found for quantum groups with a truncated tensor product in \cite{Sa}. For simply-laced types ($A,D,E$) the locations of these invertibles are identical to those in Proposition~\ref{PropInvertibles}, but the locations are different for the other types. This is because the fundamental alcove is replaced by a larger set of weights where the far wall is orthogonal to the highest root $\theta_l$ instead of $\theta_s$.
\end{remark}

\subsubsection{Table of minuscule weights}\label{TabMinuscule} Below we list all of the minuscule weights in each type, and their corresponding symmetries of the fundamental alcove in $\widehat W_p$ as in Proposition~\ref{PropInvertibles} (these were verified using a Mathematica program). In expressions for $\sigma$, $t_\lambda$ denotes the translation by the weight $\lambda$. Note that we use the labelling conventions from Table~\ref{TabGroups}, and the groups $E_8$, $F_4$ and $G_2$ do not have any non-zero minuscule weights.

\vspace{0.5em}{\noindent
\begin{tabular}{r|r|l}
Type & Minuscule $\varpi$ & $\sigma\in\widehat W_p$ with $\sigma\cdot A=A$ and $\sigma\cdot0=(p-h)\varpi$ \\\hline
$A_r$ & $\varpi_i\ \forall i$ & $t_{p\varpi_i}(s_1s_2\cdots s_r)^i$\\\hline
$B_r$ & $\varpi_r$ & $t_{p\varpi_r}w_rw_{r-1}\cdots w_1$ where $w_i=s_is_{i+1}\cdots s_r$\\\hline
$C_r$ & $\varpi_1$ & $t_{p\varpi_1}s_1s_2\cdots s_{r-1}s_rs_{r-1}\cdots s_1$\\\hline
\multirow{4}{*}{$D_r$} & $\varpi_1$ & $t_{p\varpi_1}s_1s_2\cdots s_{r-2}s_rs_{r-1}\cdots s_1$\\
& $\varpi_{r-1}$ & $t_{p\varpi_{r-1}}s_{r-1}w^+_{r-2}w^-_{r-3}w^+_{r-4}w^-_{r-5}\cdots w^\pm_1$\\
& $\varpi_r$ & $t_{p\varpi_r}s_rw^-_{r-2}w^+_{r-3}w^-_{r-4}w^+_{r-5}\cdots w^\mp_1$\\
& & where $w^+_i=s_is_{i+1}\cdots s_{r-2}s_r$ and $w^-_i=s_is_{i+1}\cdots s_{r-2}s_{r-1}$\\\hline
\multirow{2}{*}{$E_6$} & $\varpi_1$ & $t_{p\varpi_1}s_1s_2s_3s_4s_5s_6s_3s_2s_4s_3s_6s_1s_2s_3s_4s_5$\\
& $\varpi_5$ & $t_{p\varpi_5}s_5s_4s_3s_2s_1s_6s_3s_2s_4s_3s_6s_5s_4s_3s_2s_1$\\\hline
$E_7$ & $\varpi_1$ & $t_{p\varpi_1}s_1s_2s_3s_4s_5s_6s_7s_4s_5s_3s_4s_2s_1s_7s_3s_2s_4s_3s_5s_4s_7s_6s_5s_4s_3s_2s_1$\\\hline
\end{tabular}}\vspace{0.5em}

\begin{prop}\label{PropVerRep}
Suppose $G$ is an adjoint-type group. If $\fg$ is the image of ${\rm Lie}(G)$ in $\Ver_p^+$ and $\phi:L_2\to\fg$ is the constraint coming from a principal $\mathfrak{sl}_2$ of ${\rm Lie}(G)$, then we have $\Ver_p(G)\simeq\Rep_{\Ver_p^+}(\fg,\phi)$. Moreover, $\Ver_p(G)$ has no non-trivial proper tensor subcategories.
\end{prop}

\begin{proof}
First we show that the object $\fg\in\Rep(\fg,\phi)$ with the adjoint representation generates $\Rep(\fg,\phi)$. By \cite[\S 7.1]{Ve1}, representations of $\fg$ in $\Ver_p^+$ are precisely representations of the affine group scheme that corresponds to the Harish-Chandra pair $(1,\fg)$, with $1$ the trivial group. Combining this with \cite[Theorem 4.3.1]{CEO3}, the tensor subcategory of $\Rep(\fg,\phi)$ generated by $\fg$ is of the form $\Rep(\fg/\fg',\phi)$ for some Lie algebra ideal $\fg'\subseteq\fg$. As seen in Table~\ref{TabImages}, the image of $\fg$ in $\Ver_p^+$ contains no summands isomorphic to $\unit$, and hence the action of the principal $\mathfrak{sl}_2\subseteq\fg$ on $\fg'$ is non-zero. So by antisymmetry, the Lie bracket restricted to $\fg'\otimes\mathfrak{sl}_2$ is non-zero, and hence the action of $\fg'$ on $\fg$ is non-zero. This contradicts $\fg\in\Rep(\fg/\fg',\phi)$ unless $\fg'=0$, hence $\Rep(\fg/\fg',\phi)=\Rep(\fg,\phi)$ and $\fg$ generates $\Rep(\fg,\phi)$.

The action of ${\rm Lie}(G)$ on a tilting $G$-representation maps to an action of $\fg$ on its image in $\Ver_p^+$, and since this image factors through $\Tilt\SL_2$, it respects the constraint $\phi$. This gives a tensor functor $F:\Ver_p(G)\to\Rep(\fg,\phi)$, and since the image of this functor contains $\fg$ it must be surjective. Let $\pi$ be the fundamental group of $\Ver_p(G)$ considered as an affine group scheme in $\Ver_p^+$, and $(\pi_0,\fg_\pi)$ the corresponding Harish-Chandra pair by the theory of \cite{Ve1}. By \cite[Theorem 4.2.1]{CEO3}, $F$ corresponds to an injective morphism $(1,\fg)\to(\pi_0,\fg_\pi)$. Since $\fg$ is an object in $\Ver_p(G)$, we have an action of $\fg_\pi$ on $\fg$ commuting with the Lie bracket, and hence $\fg$ is an ideal of $\fg_\pi$. This means $(1,\fg)$ is normal in $(\pi_0,\fg_\pi)$, so by \cite[Theorem 4.3.1]{CEO3} again we have tensor subcategory $\Rep((\pi_0,\fg_\pi/\fg),\phi)$ of $\Ver_p(G)$. Now $\fg$ cannot be in this subcategory since its action on itself is non-zero, but every simple except $\unit$ in $\Ver_p(G)$ generates $\fg=L(\theta_l)$ by Proposition~\ref{PropInvertibles}. So we must have $\Rep((\pi_0,\fg_\pi/\fg),\phi)\simeq\Vecc$ and hence $(\pi_0,\fg_\pi)=(1,\fg)$, so $\Rep(\fg,\phi)\simeq\Ver_p(G)$. Finally, since every simple generates $\fg$ and $\fg$ generates $\Rep(\fg,\phi)$, $\Ver_p(G)$ has no non-trivial proper tensor subcategories.
\end{proof}

\begin{proof}[Proof of Theorem~\ref{ThmVerDecomp}]
The relation between invertibles and cosets in Proposition~\ref{PropInvertibles} shows that $\Ver_p^+(G)$ contains no invertibles except $\unit$. Since the summands in $\Rep G$ of $L(\lambda)\otimes L(\mu)$ have highest weights in $\lambda+\mu+\mZ\Phi$ for all $\lambda,\mu\in X(T)^+$, the correspondence $\lambda+\mZ\Phi\mapsto[L(\lambda)]$ between $X(T)/\mZ\Phi$ and the group of invertible objects in the Grothendieck ring is a group isomorphism. Thus we have an equivalence $\Ver_p^0(G)\simeq\sVec_{X(T)/\mZ\Phi,z}$ by Lemma~\ref{LemPointed} with $z$ the map in \ref{SecVerSetup}, and these categories are equivalent to $\Rep(Z,z)$ as described in \ref{SecZG}, proving (1). (2) and (3) are proven in Proposition~\ref{PropInvertibles} and Proposition~\ref{PropVerRep}, and (4) then follows from Lemma~\ref{LemDecomp}.
\end{proof}

\section{Categories generated by an object with finite powers}\label{SecMainThm}

Fix an algebraically closed field $\bk$ with characteristic $p>0$. We can define $\Ver_p(G)$ for reductive algebraic groups $G$ similarly to simple groups, and in particular we will be interested in $\Ver_p(\GL_n)$.

\subsection{Reducing Theorem~\ref{Thmmnp}(1) to a question on Lie algebras}

\begin{lemma}\label{LemSchurWeyl}
Let $\cC$ be a tensor category generated by $X\in\cC$ with $\mathbf N(X)=p$, and let $n=n(X)$. Then there are tensor functors
$$\Ver_p(\GL_n)\;\xrightarrow{H}\;\cC\;\xrightarrow{F}\;\Ver_p$$
sending the natural $n$-dimensional representation $V$ of $\GL_n$ to $X\in\cC$ and $L_{n-1}\in\Ver_p$ respectively.
\end{lemma}

\begin{proof}
The oriented Brauer category $OB_n$ is the universal $\bk$-linear rigid symmetric monoidal category generated by an object $U$ with categorical dimension $n$, meaning there exists a $\bk$-linear symmetric monoidal functor $OB_n\to\cC$ sending $U$ to $X$, and a similar functor $OB_n\to\Rep\GL_n$ sending $U$ to the $n$-dimensional natural representation $V$. Since $n\leq p-2$, it is a consequence of Schur-Weyl duality that the latter functor is full, and the collection of morphisms sent to zero is exactly the tensor ideal of $OB_n$ generated by the $(n+1)$-th antisymmetriser $a_{n+1}$ of $U$ as defined in Lemma~\ref{LemExtDims}. This means that $\Tilt\GL_n$ (which is generated by $V$) is the universal $\bk$-linear Karoubian rigid symmetric monoidal category generated by an object $V$ with $\Lambda^{n+1}V=0$, and hence we have an additive symmetric monoidal functor $\Tilt\GL_n\to\cC$ sending $V$ to $X$.

In $\Tilt\GL_n$ for $n\leq p-2$, we claim that the ideal of negligible morphisms is generated by $\Sym^{p-n+1}V$. Since all non-negligible objects are simple, the tensor ideal is thick (meaning it is generated as a tensor ideal by identity morphisms), and we can focus on objects. The representation $\Sym^{p-n+1}V$ is a simple (tilting) representation with highest weight lying on the far wall of the fundamental alcove $A$. That all other such simple (tilting) representations appear as direct summands of $(\Sym^{p-n+1}V)\otimes V^{\otimes i}$ follows from character theory and the corresponding property in characteristic zero (which is a consequence of Pieri's formula). That this entire collection of simple tilting representations on the far wall generate the ideal of negligible objects follows from translation functors (see \cite[\S E1]{Jantzen}). Hence the above functor induces a tensor functor from the semisimplification $\Ver_p(\GL_n)$ to $\cC$. Then it follows for instance from \cite[Theorem~8.2.1]{CEO3} that $\cC$ is Frobenius exact, and hence by \cite{CEO1}, there exists a (unique) tensor functor $F:\cC\to\Ver_p$.
\end{proof}

\begin{lemma}\label{LemGLn}
For $n<p$, let $V$ be the natural $n$-dimensional representation of $\GL_n$ and $\varpi_1$ the highest weight of $V$ restricted to $\SL_n\subset\GL_n$. Let $V^+$ be the unique simple object in $\Ver_p(\PGL_n)$ whose highest weight is in the $\widehat W_p$ orbit of $\varpi_1$ (see Proposition~\ref{PropInvertibles}). We have an equivalence
$$\Ver_p(\GL_n)\simeq\Ver_p(\PGL_n)\boxtimes\sVec_{\mZ,z}$$
sending $V$ to $V^+\boxtimes V_0$, where $V_0$ is the generator of $\sVec_{\mZ,z}$, and $z$ is such that $V_0$ is even if $n$ is odd and odd if $n$ is even.
\end{lemma}

\begin{proof}
Restriction $\Tilt(\GL_n)\to\Tilt(\SL_n)$ is essentially surjective, since each category is generated by $V$ and we have an isomorphism $\End_{\GL_n}(V^{\otimes r})\cong\End_{\SL_n}(V^{\otimes r})$ for all $r\in\mN$. Thus we have functors
$$\Ver_p(\PGL_n)\hookrightarrow\Ver_p(\GL_n)\twoheadrightarrow\Ver_p(\SL_n)$$
which compose to give the inclusion $\Ver_p(\PGL_n)=\Ver_p^+(\SL_n)\hookrightarrow\Ver_p(\SL_n)$. If $L$ is the invertible object in $\Ver_p(\SL_n)$ satisfying $V\cong V^+\otimes L$, then by surjectivity there is an invertible object $L'\in\Ver_p(\GL_n)$ such that $V=V^+\otimes L'$. If $\cD$ is the subcategory generated by $L'$, then we have $\Ver_p(\GL_n)\simeq\Ver_p(\PGL_n)\boxtimes\cD$ by Lemma~\ref{LemDecomp} and Theorem~\ref{ThmVerDecomp}. Since $\Ver_p(\GL_n)$ has infinitely many invertible objects, $L'$ cannot have finite order and hence $\cD$ must be equivalent to $\sVec_{\mZ,z}$ with $z$ determined by the parity of $L'$. This must match $L$ which is even if $n$ is odd and odd if $n$ is even, as described in Example~\ref{ExSLn}.
\end{proof}

\begin{prop}\label{Propmnp}
Let $\cC$ be a tensor category generated by $X\in\cC$ with $\Lambda^nX\cong\unit\cong\Sym^{p-n}X$. Then $n$ is odd, and there is a Lie algebra $\fg$ in $\Ver_p^+$ which is a subalgebra of the Lie algebra $\mathfrak{sl}(L_{n-1})$ defined in Example~\ref{ExSLn} containing the summand $\phi:L_2\hookrightarrow\mathfrak{sl}(L_{n-1})$ such that $\cC\simeq\Rep(\fg,\phi)$.
\end{prop}

\begin{proof}
Let $V,V^+,V_0$ be as in Lemma~\ref{LemGLn}, and let $X^+,X_0$ be the images in $\cC$ of $V^+,V_0$ respectively under the functor $H$ in Lemma~\ref{LemSchurWeyl}. Since $\Lambda^nX\cong\unit\cong\Sym^{p-n}X$, we have $X_0^{\otimes n}\cong\unit\cong X_0^{\otimes(p-n)}$, and then since $n$ and $p-n$ are coprime we must have $X_0\cong\unit$. This means $V_0$ is even and hence $n$ is odd. Then both $V$ and $V^+$ are sent to $L_{n-1}\in\Ver_p$, which is in $\Ver_p^+$, and hence $H$ and $F$ restrict to
$$\Ver_p^+(\SL_n)\xrightarrow{H}\cC\xrightarrow{F}\Ver_p^+.$$
We have $\Ver_p^+(\SL_n)\simeq\Rep(\mathfrak{sl}(L_{n-1}),\phi)$ by Theorem~\ref{ThmVerDecomp} and Example~\ref{ExSLn}, and then we can replace $\mathfrak{sl}(L_{n-1})$ with an affine group scheme using the methods in \cite[\S 7.1]{Ve1}. By virtue of the place of $\cC$ in the sequence of tensor functors above, it follows from the general theory of \cite[Theorem~4.2.1]{CEO3} that we can classify $\cC$ in terms of subgroup schemes, or equivalently Lie subalgebras of $\mathfrak{sl}(L_{n-1})$ containing $L_2$.
\end{proof}

\subsection{Subalgebras of $\mathfrak{sl}(L_{n-1})$}

In light of the previous section, we would like to determine all possible subalgebras of $\mathfrak{sl}(L_{n-1})$ and use these to classify the tensor categories in Proposition~\ref{Propmnp}. Since we have an isomorphism $\mathfrak{sl}(L_{n-1})\cong\mathfrak{sl}(L_{p-n-1})$, we may assume $n<p/2$ without loss of generality. Of course this means that we allow $n$ to be even.

\begin{lemma}\label{LemConditions} In $\Ver_p(\SL_2)$ with $p>2$, let $2\leq n<p/2$ and $1\leq i,j,k<n$. The restriction of the Lie bracket $\mathfrak{sl}(L_{n-1})\otimes\mathfrak{sl}(L_{n-1})\to\mathfrak{sl}(L_{n-1})$ to the summands $L_{2i}\otimes L_{2j}\to L_{2k}$ is non-zero if and only if all of the following hold:
\begin{enumerate}
\item $L_{2k}$ is a summand of $L_{2i}\otimes L_{2j}$ (that is $|i-j|\leq k\leq i+j$ and $i+j+k\leq p-2$);
\item $i+j+k$ is odd;
\item $S(n,i,j,k)\not\equiv0\mod p$, where
$$S(n,i,j,k)\coloneqq\sum_{t=\max\{i,j,k,i+j+k-n+1\}}^{\min\{i+j,i+k,j+k\}}(-1)^t{t+n\choose i+j+k+1}{i\choose t-j}{j\choose t-k}{k\choose t-i}$$
\end{enumerate}
Moreover, these conditions are invariant under permutations of $i,j,k$. The number $S(n,i,j,k)$ can alternately be written in terms of $\Delta$ functions and a $6j$-symbol (see \cite[\S A]{CS}),
$$S(n,i,j,k)=\frac{(-1)^{n-1}i!j!k!((i+j+k+1)!)^{-1}}{\Delta(i,j,k)\Delta(\frac{n-1}2,\frac{n-1}2,i)\Delta(\frac{n-1}2,\frac{n-1}2,j)\Delta(\frac{n-1}2,\frac{n-1}2,k)}\left\{\begin{matrix}i&j&k\\\frac{n-1}2&\frac{n-1}2&\frac{n-1}2\end{matrix}\right\}.$$
\end{lemma}

\begin{proof}
Let $V_a$ be the simple of highest weight $a\in\mN$ in $\Rep\SL_2$ so that we have isomorphisms
$$\Hom_{\SL_2}(V_k,V_a\otimes V_b)\xrightarrow{\sim}\Hom_{\Ver_p}(L_k,L_a\otimes L_b)$$
for $a+b+k\leq2p-4$. We will use some known properties of summands of $V_a\otimes V_b$ in characteristic zero and examine their restriction to characteristic $p$, using the isomorphism above to extrapolate to $\Ver_p$. In characteristic zero we have an inclusion $V_k\hookrightarrow V_a\otimes V_b$ if and only if $a,b,k\in\mN$ satisfy $|a-b|\leq k<a+b$ and $\chi\coloneqq\frac{a+b-k}2\in\mN$. Such an inclusion $\iota_k^{a,b}$ is given in \cite[Equation~11.2]{CS} in terms of the basis $F^ie_a\in V_a$ where $e_a$ is a highest weight vector and $F$ is the weight-lowering element in $\mathfrak{sl}_2$:
$$\iota_k^{a,b}(e_k)=\sum_{r=0}^\chi(-1)^r\frac{{\chi\choose r}}{{\chi+k\choose a-r}}F^r e_a\otimes F^{\chi-r}e_b.$$
We have $L_k\subseteq L_a\otimes L_b$ in $\Ver_p$ if and only if the above conditions on $a,b,k$ hold and also $a+b+k\leq2p-4$; if this is the case, then the binomial coefficients in the above formula are rational numbers $\frac xy$ with $p$ not dividing $x,y$, so this also applies in characteristic $p$ by taking $\frac xy\mapsto xy^{-1}$ modulo $p$. Moreover, we see from the above formula that
$$c_{V_a,V_b}\circ\iota_k^{a,b}=(-1)^\chi\iota_k^{b,a}.$$

To determine when the restriction of the Lie bracket $\mathfrak{sl}(L_{n-1})\otimes\mathfrak{sl}(L_{n-1})\to\mathfrak{sl}(L_{n-1})$ to $L_{2i}\otimes L_{2j}\to L_{2k}$ is non-zero, it suffices to evaluate the morphism $\psi^{i,j}_{k,n}:V_{2k}\to V_{n-1}\otimes V_{n-1}$ given by the composition
$$V_{2k}\hookrightarrow V_{2i}\otimes V_{2j}\hookrightarrow V_{n-1}\otimes V_{n-1}^*\otimes V_{n-1}\otimes V_{n-1}^*\xrightarrow{[\cdot,\cdot]}V_{n-1}\otimes V_{n-1}^*\xrightarrow{\sim} V_{n-1}\otimes V_{n-1}$$
as a scalar multiple of $V_{2k}\hookrightarrow V_{n-1}\otimes V_{n-1}$, where all the inclusions are the morphisms $\iota$ above. We can write $\psi^{i,j}_{k,n}$ as the sum of the compositions
$$\begin{tikzcd}[row sep=0.5em, column sep=7em]
V_{2k} \arrow{r}{\iota_{2k}^{2i,2j}} & V_{2i}\otimes V_{2j} \arrow{r}{\iota_{2i}^{n-1,n-1}\otimes\iota_{2j}^{n-1,n-1}} 
& V_{n-1}^{\otimes4} \arrow{r}{\id\otimes\mathrm{eval}\otimes\id} & V_{n-1}\otimes V_{n-1},\\
V_{2k} \arrow{r}{-c_{V_{2i},V_{2j}}\circ\iota_{2k}^{2i,2j}} & V_{2j}\otimes V_{2i} \arrow{r}{\iota_{2j}^{n-1,n-1}\otimes\iota_{2i}^{n-1,n-1}} 
& V_{n-1}^{\otimes4} \arrow{r}{\id\otimes\mathrm{eval}\otimes\id} & V_{n-1}\otimes V_{n-1}.\\
\end{tikzcd}$$
where `eval' is the morphism $V_{n-1}\otimes V_{n-1}\xrightarrow{\sim} V_{n-1}^*\otimes V_{n-1}\to\bk$. The first composition is \cite[Equation~9.5]{CS} which evaluates to $\lambda\iota_{2k}^{n-1,n-1}$, where the scalar $\lambda$ is computed in \cite[Theorem~11.2]{CS}. The second composition is $(-1)^{i+j+k}\lambda\iota_{2k}^{n-1,n-1}$ by the braid formula above. So if $i,j,k$ are such that $L_{2k}\subseteq L_{2i}\otimes L_{2j}$ in $\Ver_p$, then $\psi^{i,j}_{k,n}=0$ in characteristic $p$ if and only if either $i+j+k$ is even or the rational number $\lambda$ above becomes zero when converted to characteristic $p$. Unpacking the expression for $\lambda$ from \cite{CS}, we have
\begin{align*}
\lambda&=\frac{(-1)^k(i+j-k)!(i-j+k)!(j-i+k)!(n-1-i)!i!^2(n-1-j)!j!^2}{(n+k)(i+j+k)!(n-1+i)!(n-1+j)!}\\
&\qquad\cdot\sum_t\frac{(-1)^{t+n-1}(t+n)!}{(t-i)!(t-j)!(t-k)!(t-i-j-k+n-1)!(i+j-t)!(i+k-t)!(j+k-t)!}
\end{align*}
with $t$ ranging over all integers such that the values in the denominator factorials are non-negative. Since we require $i+j+k<p$ this expression is comprised of integers or factorials of integers less than $p$ (excluding $(t+n)!$ which is not being divided by), and hence can be converted to characteristic $p$. Whether the result is congruent to zero modulo $p$ depends only on the sum in $t$, which is invariant under permutations of $i,j,k$ and can be rearranged to give $S(n,i,j,k)$ up to a non-zero scalar.
\end{proof}

\begin{prop}\label{PropSubalgebras}\samepage
For $2\leq n<p/2$, the subalgebras of $\mathfrak{sl}(L_{n-1})=L_2\oplus L_4\oplus\cdots\oplus L_{2n-2}$ are as follows:
\begin{enumerate}[label=(\alph*)]
\item $\mathfrak{sl}(L_{n-1})$ itself;
\item $L_2\oplus L_6\oplus \cdots$ (all terms of the form $L_{4r+2}$, $r\in\mN$);
\item $L_2$;
\item $L_2\oplus L_{10}$ for $n=7$ and $p\geq17$;
\item $L_2\oplus L_{2n-2}$ for $p=2n+1\geq7$ (distinct from the above cases when $p\geq11$);
\item $L_2\oplus L_{14}$ for $n=10$ and $p=23$.
\end{enumerate}
\end{prop}

\begin{proof} Suppose $i,j,k$ satisfy conditions (1) and (2) in Lemma~\ref{LemConditions}, and assume without loss of generality that $i\geq j\geq k$. If $k=1$ then we must have $i=j$, meaning the Lie bracket $\mathfrak{sl}(L_{n-1})\otimes\mathfrak{sl}(L_{n-1})\to\mathfrak{sl}(L_{n-1})$ is zero when restricted to the summands $L_2\otimes L_{2i}\to L_{2j}$ for $i\neq j$. This implies that $L_2$ is always a subalgebra, and if $L_{2i}$ for some $i\geq2$ generates only $L_2$ then $L_2\oplus L_{2i}$ must also be a subalgebra.

Letting $d=i-j$ and $c=\frac{k-d+1}2$, we can write $S(n,i,j,k)$ as
$$S(n,i,j,k)=\frac{(-1)^i(n+i)!(i+c)!}{(i+j+k+1)!k!(n-1-i)!(j-c)!}P(n,i,j,k)$$
where $P(n,i,j,k)$ is as follows, denoting $(a)_b\coloneqq a(a-1)\cdots(a-b+1)$ and $(a)_0=1$:
\begin{align*}
P(n,i,j,k)&\coloneqq\frac1{(i+c)_c}\sum_{s=0}^{k-d}(-1)^s{k\choose s}{k\choose d+s}\frac{(i-d)_s(i-d)_{k-d-s}}{(i-d)_c}(n+i+s)_s(n-i-1)_{k-d-s}\\
&=\frac{(-1)^{i+n-1}j!k!^2(j-c)!\Delta(\frac{n-1}2,\frac{n-1}2,i)}{i!(i+c)!\Delta(i,j,k)\Delta(\frac{n-1}2,\frac{n-1}2,j)\Delta(\frac{n-1}2,\frac{n-1}2,k)}\left\{\begin{matrix}i&j&k\\\frac{n-1}2&\frac{n-1}2&\frac{n-1}2\end{matrix}\right\}.
\end{align*}
Since all terms in the expression for $S(n,i,j,k)$ besides $P$ are factorials of numbers strictly between 0 and $p$, we have $S(n,i,j,k)\equiv0$ if and only if $P(n,i,j,k)\equiv0$ mod $p$. For fixed $k,d$ the function $P(n,i,i-d,k)$ is a rational function in $n$ and $i$ which we compute explicitly in several cases:
\begin{enumerate}
\item $P(n,i,i-k+1,k)=-2k$ for all $k\geq1$;
\item $P(n,i,i,3)=4(3n^2-5i-5i^2+3)$;
\item $P(n,i,i-1,4)=8(3n^2-7i^2+15)$;
\item $P(n,i,i,5)=-4(120+225n^2+15n^4-266i-70n^2i-203i^2-70n^2i^2+126i^3+63i^4)$;
\item $P(n,i,i,7)=8(6300+16415n^2+2450n^4+35n^6-16110i-7875n^2i-315n^4i-10599i^2$ 
$-7182n^2i^2-315n^4i^2+10593i^3+1386n^2i^3+4224i^4+693n^2i^4-1287i^5-429i^6).$
\end{enumerate}

Since $p>2n>2k$, we have $P(n,i,i-k+1,k)\not\equiv0\mod p$ for all $k$ by (1). Thus the Lie bracket is non-zero on the following summands:
\begin{itemize}[leftmargin=0.6in]
\item[($k=1$)] $L_{2i}\otimes L_{2i}\to L_2$ for $i\geq1$, meaning any subalgebra of $\mathfrak{sl}_2$ must contain $L_2$.
\item[($k=2$)] $L_4\otimes L_{2(i-1)}\to L_{2i}$ for $i\geq3$, meaning $L_4$ generates $\mathfrak{sl}(L_{n-1})$.
\item[($k=3$)] $L_6\otimes L_{2(i-2)}\to L_{2i}$ for $i\geq5$, meaning $L_6$ generates all summands of the form $L_{4r+2}$.
\end{itemize}
We aim to show that every summand $L_{2i}$ for $i\geq4$ generates $L_6$ and hence all summands of the form $L_{4r+2}$, excluding the exceptional cases in the proposition. 

First we consider characteristic $p=2n+1$. For $i=n-1$ we have the subalgebra $L_2\oplus L_{2n-2}$ since the Lie bracket on $L_{p-3}\otimes L_{p-3}$ factors through $\Lambda^2 L_{p-3}\cong\Sym^2 L_1\cong L_2$. If instead $i<n-1$ (implying $n>2$ and thus $p>5$) then $P(\frac{p-1}2,i,i,3)=3p^2-5p-5(2i-1)(2i+3)$ is non-zero modulo $p$, meaning $L_{2i}\otimes L_{2i}\to L_6$ is non-zero and $L_{2i}$ generates $L_6$. Thus from here on we can assume $p\geq2n+3$ and hence $2i+3\leq p-2$ for all $i$, meaning $L_6\subseteq L_{2i}\otimes L_{2i}$ for $i\geq4$.

Next consider $L_{10}$, i.e. $i=5$. We have $P(n,5,5,3)=12(n+7)(n-7)$, so for $n=7$ the bracket morphism on $L_{10}\otimes L_{10}\to L_6$ is zero and we have the subalgebra $L_2\oplus L_{10}\subset\mathfrak{sl}(L_6)$. For all other $n>5$ we have $P(n,5,5,3)\not\equiv0$ (note that $n+7<p$), so the summand $L_{10}$ generates $L_6$ and hence all $L_{4r+2}$.

Next consider $L_8$, i.e. $i=4$. We have $P(n,4,4,3)=-88$ for $n=5$ and $200$ for $n=7$, both of which have no prime factor $p\geq2n+3$ meaning $L_8\otimes L_8\to L_6$ is non-zero. If instead $n=6$ or $n\geq8$ then $p$ is so large that $L_{10}\subseteq L_8\otimes L_8$, so in this case we consider $2P(n,5,4,4)-P(n,4,4,3)=504$ which has no prime factor $p\geq 2n+3>13$. This means at least one of $P(n,4,4,3),P(n,5,4,4)$ is non zero so $L_8$ generates either $L_6$ or $L_{10}$, and then since $n\neq 7$, $L_{10}$ generates $L_6$.

For $i>5$ we can assume $2i+5\leq p-2$ meaning $L_{10}\subseteq L_{2i}\otimes L_{2i}$, since otherwise we would have $p=2n+3=2i+5$ meaning $P(n,i,i,3)=-2p^2+22p-36\not\equiv0\mod p$. We can also assume $n>7$ to avoid the exceptional subalgebra $L_2\oplus L_{10}$ above, since the only case with $n=7$ is $i=6$ which has $P(7,6,6,3)=-240\not\equiv0\mod p$. We have
\begin{align*}
P(n,i,i,5)+5(n^2-3i^2-3i+14)P(n,i,i,3)&=12(2i-1)(2i+3)Q_1(i),\\
25P(n,i,i,5)+7(23n^2-45i^2-45i+163)P(n,i,i,3)&=36(2n-1)(2n+1)Q_2(n)
\end{align*}
where $Q_1(i)\coloneqq i^2+i-10$ and $Q_2(n)\coloneqq3n^2-47$. Thus if $Q_1(i)\not\equiv0$ or $Q_2(n)\not\equiv0\mod p$ then $L_{2i}$ generates $L_6$ in $\mathfrak{sl}(L_{n-1})$.

For $i=6$ we have $Q_1(6)=2^5\not\equiv0\mod p$, so $L_{12}$ generates $L_6$. For $L_{14}$, i.e. $i=7$, we have $Q_1(7)=2\cdot23$ which is non-zero unless $p=23$, and $Q_2(n)\equiv3(n-10)(n+10)\mod 23$ which is non-zero unless $n=10$. We in fact have a subalgebra $L_2\oplus L_{14}\subset\mathfrak{sl}(L_9)$ in characteristic 23, since $P(10,7,7,5)\equiv P(10,7,7,3)\equiv0\mod 23$ and $L_{18}\not\subseteq L_{14}\otimes L_{14}$ in $\Ver_{23}$. If $p\neq23$ or $n\neq 10$ then $L_{14}$ generates $L_6$.

Finally, we show that $L_{2i}$ for any $i\geq8$ generates at least one of $L_6$, $L_{10}$ or $L_{14}$. We can assume $2i+7\leq p-2$ meaning $L_{14}\subseteq L_{2i}\otimes L_{2i}$, since otherwise we would have $i=\frac{p-7}2$ meaning $4Q_1(i)=p^2-12p-5\not\equiv0\mod p$. We can also assume $n>10$ in characteristic 23 to avoid the exceptional subalgebra $L_2\oplus L_{14}$ above, since $Q_1(i)\not\equiv0\mod 23$ for $i=8,9$.
\begin{align*}
27P(n&,i,i,7)+Q_3(n,i)Q_1(i)+Q_4(n,i)Q_2(n)=8465600=2^6\cdot5^2\cdot11\cdot13\cdot37,\text{ where}\\
Q_3&\coloneqq648(1300+315n^2+105n^4-407i-231n^2i-264i^2-231n^2i^2+286i^3+143i^4),\\
Q_4&\coloneqq-280(1180-39n^2+9n^4).
\end{align*}
Thus to have $P(n,i,i,7)\equiv Q_1(i)\equiv Q_2(n)\equiv0\mod p$ we require $p=37$, but then $Q_2(n)\equiv3(n+18)(n-18)$ which has no solutions with $37\geq2n+3$, so $L_{2i}$ generates $L_6$. Taking unions of the subalgebras found produces no new subalgebras, and we are done.
\end{proof}

\subsection{Proof of Theorem~\ref{Thmmnp}} For each $n$ and each of the subalgebras $\fg'\subseteq\mathfrak{sl}(L_{n-1})$ in Proposition~\ref{PropSubalgebras}, we can use Section~\ref{SecWeylFactors} (and Proposition~\ref{PropTypeD} in case (e)) to find an algebraic group $G$ such that
\begin{enumerate}
\item $G$ has a representation whose image in $\Ver_p^+$ is $L_{n-1}$, and
\item the image $\fg$ of the Lie algebra of $G$ in $\Ver_p^+$ has the same underlying object as $\fg'$.
\end{enumerate}
Property (1) means that the category $\Ver_p(G)\simeq\Rep(\fg,\phi)$ as in Theorem~\ref{ThmVerDecomp} satisfies the conditions of Proposition~\ref{Propmnp}, so $\fg$ is isomorphic to some Lie subalgebra of $\mathfrak{sl}(L_{n-1})$, and then (2) means this subalgebra must be $\fg'$. Thus we have described the representation category of every Lie subalgebra of $\mathfrak{sl}(L_{n-1})$, and thus we obtain the following theorem as a result of Proposition~\ref{Propmnp}.

\begin{theorem}\label{Thmmnp1}\samepage
Let $\cC$ be a tensor category generated by $X\in\cC$ with $\mathbf{N}(X)=p$ and suppose $\Lambda^nX\cong\unit\cong\Sym^{p-n}X$. Then $n$ is odd and we have one of the following possibilities:
\begin{enumerate}[label=(\alph*)]
\item There is an equivalence $\cC\simeq\Ver_p(\PGL_n)$ sending $X$ to $L(\lambda)$, where $\lambda$ is in the $\widehat W_p$-orbit of $\varpi_1$, the highest weight of the natural $n$-dimensional $\SL_n$-representation.
\item There is an equivalence $\cC\simeq\Ver_p(\SO_n)$ sending $X$ to $L(\varpi_1)$, the natural $n$-dimensional $\SO_n$-representation.
\item There is an equivalence $\cC\simeq\Ver_p^+$ sending $X$ to $L_{n-1}$.
\item $n=7$, $p\geq13$, and there is an equivalence $\cC\simeq\Ver_p(G_2)$ sending $X$ to $L(\varpi_1)$, the 7-dimensional simple representation.
\item $p\geq7$, either $p=2n+1$ or $p=2n-1$, and there is an equivalence $\C\simeq\Ver_p(\PSO_{p-1})$ sending $X$ to $L(\lambda)$, where $\lambda$ is in the $\widehat W_p$-orbit of the highest weight of a half-spin representation $\varpi_{\frac{p-1}2}$ or $\varpi_{\frac{p-3}2}$.
\item $n=13$, $p=23$, and there is an equivalence $\C\simeq\Ver_{23}^+(E_7)$ sending $X$ to $L(4\varpi_1)$. Note that $4\varpi_1$ is in the $\widehat W_p$-orbit of $\varpi_1$, the highest weight of the 56-dimensional simple representation.
\end{enumerate}
\end{theorem}

As an additional result, we also complete a list of equivalences, see e.g. \cite[\S 4]{CEO2}.

\begin{prop}\label{CorEquivalences}\samepage
The following is a complete list of Verlinde categories of distinct adjoint-type simple algebraic groups that are equivalent.
\begin{enumerate}
\item $\Ver_p(G)\simeq\Ver_p(G')\simeq\Vecc$ when the Coxeter numbers of $G$ and $G'$ are both $p-1$.
\item $\Ver_p(\PGL_n)\simeq\Ver_p(\PGL_{p-n})$ for $n\geq2$ and $p\geq n+2$.
\item $\Ver_p(\Sp^\ad_{2n})\simeq\Ver_p(\SO_{p-2n})$ for $n\geq2$ and $p\geq 2n+3$.
\item $\Ver_p(\Sp^\ad_{p-3})\simeq\Ver_p(\SO_{p-2})\simeq\Ver_p(\PGL_{p-2})\simeq\Ver_p(\PGL_2)=\Ver_p^+$ for $p\geq7$.
\item $\Ver_{11}(G_2)\simeq\Ver_{11}(\Sp^\ad_8)\simeq\Ver_{11}(\SO_9)\simeq\Ver_{11}(\PGL_9)\simeq\Ver_{11}(\PGL_2)=\Ver_{11}^+$.
\item $\Ver_{13}(G_2)\simeq\Ver_{13}(\PSO_{12})$.
\item $\Ver_{17}(F_4)\simeq\Ver_{17}(\PSO_{16})$.
\item $\Ver_{19}(F_4)\simeq\Ver_{19}(G_2)$.
\end{enumerate}
\end{prop}

\begin{proof}

Let $G,G'$ be adjoint-type groups and let $\fg,\fg'$ be the images of their Lie algebras in $\Ver_p^+$. Equivalence (1) is a consequence of Proposition~\ref{PropInvertibles}, and equivalences (4) and (5) are cases where $\fg=L_2$ and hence $\Ver_p(G)\simeq\Ver_p^+$. Using Table~\ref{TabImages}, it is straightforward to list all cases where $\fg$ and $\fg'$ have the same underlying object in $\Ver_p^+$. When this occurs, it suffices to find $2\leq n\leq p-3$ and representations $X,X'$ of $G,G'$ which both have image $L_{n-1}\in\Ver_p^+$. Indeed, no two distinct Lie subalgebras of $\mathfrak{sl}(L_{n-1})$ have the same underlying object by Proposition~\ref{PropSubalgebras}, so $\Ver_p(G)$ and $\Ver_p(G')$ must be equivalent by Proposition~\ref{PropVerRep}.

Equivalences (2), (3), (6) and (7) have objects $X$ given by ($\widehat W_p$-orbits of) the simples in Table~\ref{TabImages}. For equivalence (8), the object $X$ in $\Ver_{19}(F_4)$  is $L(3\varpi_1)$ which has image $L_6\in\Ver_p^+$, the same as $L(\varpi_1)\in\Ver_p(G_2)$. This can be verified computationally, e.g. using the LiE software \cite{LiE} and the method in Section~\ref{SecWeylFactors}. Finally, for $G=\SO_{2n+1},G'=\Sp_{2n}$ and $p\geq 2n+5$ we have $\fg\cong\fg'$ as objects in $\Ver_p^+$, but this is not a Lie algebra isomorphism. To show this, we can assume without loss of generality that $p>4n$ using equivalence (3), so that the adjoint representations of $G$ and $G'$ contain no negligible components when restricted to $\SL_2$. This means that any Lie algebra isomorphism $\fg\cong\fg'$ would lift to an isomorphism $\mathfrak{so}_{2n+1}\simeq\mathfrak{sp}_{2n}$ of the classical Lie algebras, a contradiction.
\end{proof}

\begin{remark}\label{RemOSp}
The Lie subalgebras (a), (b), (c) and (d) in Proposition~\ref{PropSubalgebras} also appear as subalgebras of $\mathfrak{sl}_n$ in characteristic zero (for example, (b) is the natural inclusion $\mathfrak{so}_n\hookrightarrow\mathfrak{sl}_n$). Cases (e) and (f) on the other hand are specific to certain characteristics, and they are the Lie algebras $D_2^*$ and $E_2^*$ in \cite{CEO2}. The appearance of $E_2^*$ as a subalgebra of $\mathfrak{sl}(L_9)\cong\mathfrak{sl}(L_{12})$ is explained by Section~\ref{SecWeylFactors}. Alternatively, the action of a principal nilpotent on the 56-dimensional $E_7$-representation $L(\varpi_1)$ has Jordan blocks of size 23, 23 and 10 as found in \cite{St}, and hence $E_2^*$ (which is the image of the Lie algebra of $E_7$) acts on the image of the 10-dimensional $\SL_2$-representation $L_9$. We provide a conceptual explanation for the appearance of $D_2^*$ below, following the reasoning of \cite[\S 4.3]{CEO2}.

The orthosymplectic Lie superalgebra $\mathfrak{osp}(1|2)$ has even component $\mathfrak{sl}_2$ and odd component the natural 2-dimensional representation of $\mathfrak{sl}_2$. This means that for $p\geq5$, we can interpret $\mathfrak{osp}(1|2)$ as a Lie superalgebra object in $\Tilt\SL_2$ which has an image in $\Ver_p$ with even part $L_2$ and odd part $L_1$. Using the inclusion $\sVec\hookrightarrow\Ver_p$ and $L_1\otimes L_{p-2}\cong L_{p-3}$, this image can be interpreted as the (non-super) Lie algebra $D_2^*=L_2\oplus L_{p-3}$ with a $\mZ/2$-grading. Now, there is a representation of $\mathfrak{osp}(1|2)$ with components given by the simple objects in $\Tilt\SL_2$ with highest weights $\frac{p-3}2$ and $\frac{p-1}2$ respectively. By semisimplifying and tensoring the odd component by $L_{p-2}$, we obtain a graded representation of $D_2^*$ with underlying object $(L_{\frac{p-3}2})^{\oplus 2}$. The action of $L_{p-3}\subset D_2^*$ on this representation is described by a $2\times 2$-matrix with values in $\bk$, and the grading means this matrix must be antidiagonal and hence diagonalisable.  This means that, after forgetting the grading, the representation $(L_{\frac{p-3}2})^{\oplus 2}$ splits into 2 simple representations with underlying object $L_{\frac{p-3}2}$. These satisfy the condition $\mathbf{N}(L_{\frac{p-3}2})=p$ so that $D_2^*$ appears in Theorem~\ref{Thmmnp1}.

The connection between representations of $D_2^*$ and $\Ver_p(\PSO_{p-1})$ is a special case of a more general level-rank duality between $\mathfrak{osp}$ and $\mathfrak{so}$ in $\Ver_p^+$. Specifically, a subalgebra $\mathfrak{so}_{2n-1}\subset\mathfrak{so}_{2n}$ splits $\mathfrak{so}_{2n}$ into a $\mZ/2$-graded Lie algebra, and then we have an isomorphism
$$\mathfrak{so}_{2n}\cong\mathfrak{osp}(1|p-2n+1)\text{ in }\Ver_p^+$$
extending the type $B$ and $C$ level-rank duality $\mathfrak{so}_{2n-1}\cong\mathfrak{sp}_{p-2n+1}$ in Proposition~\ref{CorEquivalences}. We expect that this induces an equivalence
$$\Ver_p(\PO_{2n})\simeq\Ver_p(\PSO_{2n})^{\mZ/2}\simeq\Ver_p^+(\SOSp(1|p-2n+1)),$$
where the superscript $\mZ/2$ refers to an equivariantization in the sense of \cite{DGNO}. However, since tilting modules for supergroups are not yet well studied in the literature, we leave this without proof.
\end{remark}

The following theorem completes the proof of Theorem~\ref{Thmmnp}.

\begin{theorem}\label{Thmmnp3}
Let $\cC$ be a tensor category generated by $X\in\cC$ with $\mathbf{N}(X)=p$ and denote $n=n(X)$. Then $\cC\simeq\cC^+\boxtimes\cC_0$ where $\cC^+,\cC_0$ are generated by objects $X^+,X_0$ respectively such that $X\cong X^+\otimes X_0$, $X_0$ is invertible, and $\cC^+$ is one of the categories in Theorem~\ref{Thmmnp1}. Moreover,
\begin{enumerate}
\item if $n$ is odd then $\Lambda^n X^+\cong\unit\cong\Sym^{p-n}X^+$, $X_0$ is even, and  $\cC_0$ is equivalent to $\Vecc_{\Gamma}$ for some cyclic group $\Gamma$;
\item if $n$ is even then $\Lambda^{p-n} X^+\cong\unit\cong\Sym^nX^+$, $X_0$ is odd, and $\cC_0$ is equivalent to $\sVec_{\Gamma,z}$ for some cyclic group $\Gamma$ and $z:\Gamma\to\mZ/2$.
\end{enumerate}
\end{theorem}

\begin{proof}
Let $V^+$ and $V_0$ be as in Lemma~\ref{LemGLn}, and let $X^+$ and $X_0$ be their images under the functor $H$ in Lemma~\ref{LemSchurWeyl}. Then the subcategory $\cC^+$ generated by $X^+$ is one of the categories in Theorem~\ref{Thmmnp1}, and the subcategory $\cC_0$ generated by $X_0$ is as described in the theorem. The result now follows from Lemma~\ref{LemDecomp} and the properties of $\cC^+$ from Theorem~\ref{ThmVerDecomp}.
\end{proof}

\begin{remark} Theorem~\ref{Thmmnp3} provides further evidence towards \cite[Conjecture~4.1]{CEO2}, in that all semisimple symmetric tensor categories which appear in the classification are in the list conjectured to be complete {\it loc. cit.} Moreover, Theorem~\ref{Thmmnp3} and its proof answer \cite[Question~6.5]{EO} in the affirmative for $G=\SL_n$.
\end{remark}

\newpage

\section{Related results and conjectures}\label{SecExtra}
Let $\bk$ again be an algebraically closed field of characteristic $p>0$.

\subsection{Symmetric and exterior powers in $\Ver_{p^n}$}

\subsubsection{} Let $T_n$ be the indecomposable $\SL_2$ tilting module of weight $n$. For each pair $a,b\in\{p^{n-1}-1,\dots,p^n-2\}$ choose an isomorphism $\phi_{a,b}:T_a\otimes T_b\to\bigoplus_i T_i$ expressing $T_a\otimes T_b$ as a direct sum of indecomposable tilting modules. The category $\Ver_{p^n}$ can then be explicitly described by computing the associativity and braiding morphisms on the projective modules with respect to this basis, and evaluating tensor products of morphisms by converting to $\SL_2$ and splitting over the direct sums. For example, the braiding morphism $c_{a,b}:T_a\otimes T_b\to T_b\otimes T_a$ is given by $\phi_{b,a}c_{a,b}\phi_{a,b}^{-1}$ restricted to summands with weight at most $p^n-2$, written as a matrix with entries morphisms between projectives in $\Ver_{p^n}$.

In small cases this allows symmetric and exterior powers to be computed directly, although since this description is non-strict, care has to be taken with the associativity morphisms. To compute a braiding on two objects in a tensor power $X^{\otimes n}$ one can iteratively tensor $c_{X,X}$ on the right and left with identity morphisms while rearranging to parenthesise objects in the form $((X\otimes X)\otimes X)\otimes\dots$. This means tensoring on the right simply involves decomposing $X_i\otimes X$ for each summand $X_i$, while tensoring on the left requires the additional step of applying a sequence of associativity morphisms, for example:
\begin{align*}
&X((((XX)X)X)X)\xrightarrow{a_{X,((XX)X)X),X}}(X(((XX)X)X))X\xrightarrow{a_{X,(XX)X,X}\otimes\id_X}((X((XX)X))X)X\\
&\qquad\xrightarrow{a_{X,XX,X}\otimes\id_X^{\otimes2}}(((X(XX))X)X)X\xrightarrow{a_{X,X,X}\otimes\id_X^{\otimes3}}(((XX)X)X)X)X
\end{align*}
where $a_{X,Y,X}$ is the direct sum of $a_{X,X_i,X}$ over summands $X_i$ of $Y$. The categories $\Ver_{p^n}$ for $p^n=4,8,9$ are small enough that this method is computationally feasible. Below we give results of computations performed in Mathematica regarding symmetric and exterior algebras in these categories, and propose some questions based on these results.

\subsubsection{}\label{512} Let $L_i$ for $0\leq i<(p-1)p^{n-1}$ be the simples of $\Ver_{p^n}$ as in \cite[Theorem 4.42]{BEO}, so that $L_{jp}$ is the image of $L_j$ in $\Ver_{p^{n-1}}\hookrightarrow\Ver_{p^n}$.
\begin{itemize}\samepage
\item In $\Ver_4$:
\begin{itemize}[label={}]
	\item $m(L_1)=2$ and $n(L_1)=2$, with $\Sym^2L_2\cong\Lambda^2L_2\cong L_0$;
\end{itemize}
\item In $\Ver_8$:
\begin{itemize}[label={}]
	\item $m(L_1)=6$ and $n(L_1)=2$, with $\Sym^6L_1\cong\Lambda^2L_1\cong L_0$;
	\item $m(L_3)=4$ and $n(L_3)=4$, with $\Sym^4L_3\cong\Lambda^4L_3\cong L_0$.
\end{itemize}
\item In $\Ver_9$:
\begin{itemize}[label={}]
	\item $m(L_1)=7$ and $n(L_1)=2$, with $\Sym^7L_1\cong L_3$ and $\Lambda^2L_1\cong L_0$;
	\item $m(L_2)=6$ and $n(L_2)=3$, with $\Sym^6L_2\cong\Lambda^3L_2\cong L_0$;
	\item $m(L_4)=2$ and $n(L_4)=7$, with $\Sym^2L_4\cong\Lambda^7L_4\cong L_0$;
	\item $m(L_5)=3$ and $n(L_5)=6$, with $\Sym^3L_5\cong L_3$ and $\Lambda^6L_5\cong L_0$.
\end{itemize}
\end{itemize}
Moreover, for all of the above the divided and symmetric powers are isomorphic, as are the exterior powers and dual exterior powers.

\begin{conjecture}\label{ConjVerpn}
If $n\geq2$ and $L_i$ is a simple in $\Ver_{p^n}$ not in the image of $\Ver_{p^{n-1}}\hookrightarrow\Ver_{p^n}$, then $\mathbf{N}(L_i)=p^n$.
\end{conjecture}

\subsubsection{}For an object $X$ in a tensor category with $\mathbf{N}(X)<\infty$, we recall from \cite[Definition~5.3]{Et} that $\Sym X$ is called almost Koszul if the complexes formed by the morphisms
$$\Lambda^nX\otimes\Sym^m X\to\Lambda^{n-1}X\otimes\Sym^{m+1}X$$
are exact for all $m,n$ except for $m+n\in\{0,\mathbf{N}(X)\}$. We have verified computationally that all non-invertible simple objects in $\Ver_{p^n}$ for $p^n\in\{4,8,9\}$ have almost Koszul symmetric and exterior algebras, and $\Sym L_1$ for $L_1\in\Ver_{p^n}$ is shown to be almost Koszul for all $p,n$ in \cite[\S 6]{BE}. Moreover we have verified that $E_1^*$, the length 2 object in $\Ver_9$ with $\unit$ as a subobject and $L_4$ as a quotient (see \cite[Example 10.2.5]{CEO3}), has an almost Koszul symmetric algebra with $m(E_1^*)=4$ and $n(E_1^*)=5$.

\begin{question}\label{QueKoszul}
Are all the indecomposable objects in $\Ver_{p^n}$ Koszul or almost Koszul for all $p$ and $n$?
\end{question}

\subsection{Possible values of $\mathbf{N}(X)$}

Let $\cC$ be a tensor category over $\bk$.

\subsubsection{} Let $X\in \cC$ be non-zero in a tensor category $\cC$ over an algebraically closed field of characteristic $p$. By \cite[Lemma~5.4.3]{CE}, we have $\mathbf{N}(X)\ge 4$, and by Lemma~\ref{LemDim}, we know that $\mathbf{N}(X)$ is a multiple of $p$. For $p\geq5$ all multiples of $p$ can occur as values of $\mathbf{N}(X)$, for instance $\mathbf{N}(L^{\oplus k})=pk$ for all positive integers $k$ and any non-invertible simple $L$ in $\Ver_p$. For $p\in\{2,3\}$, we cannot have $\mathbf{N}(X)=p$, but all multiples of $p^2$ occur, by \ref{512}. Conversely, \cite[Conjecture~1.4]{BEO} and Conjecture~\ref{ConjVerpn} predict that $\mathbf{N}(X)$ must always be a multiple of $p^2$ for $p\in\{2,3\}$. Below we prove this, thus providing some evidence towards both conjectures. The main observation that we use is that, for $p\in\{2,3\}$, the Koszul complex in degree $p$ must be exact (contrary to the case $p>3$). We refer to \cite[\S 3.4]{EHO} for the definition of the Koszul complex.

\begin{prop}\label{PropKoszul}
Let $X\in\cC$ be an object with $\mathbf{N}(X)<\infty$, and $r\in\mN$. If the Koszul complex for $X$ is exact in degree $d$ for all $d\leq p^r$, that is
$$0\to \Lambda^d X\to X\otimes \Lambda^{d-1}X\to\cdots\to X\otimes\Sym^{d-1}X\to \Sym^dX\to 0$$
is exact for all $0<d\le p^{r}$, then $\mathbf{N}(X)$ is divisible by $p^{r+1}$.
\end{prop}
\begin{proof}Set $m=m(X)$ and $n=n(X)$.
We use the theory of $p$-adic dimensions from \cite{EHO}. It follows from the basic expressions that $\Dim_+(X)=-m$ and $\Dim_-(X^\ast)=n$, so that all categorical dimensions of symmetric and exterior powers
$$s_j:=\dim(\Sym^j X)\quad\mbox{and}\quad a_j:=\dim(\Lambda^j X)$$ are determined by
$$(1-t)^{m}\;=\;\sum_j s_j t^j\quad\mbox{and}\quad (1+t)^n\;=\;\sum_{j}a_jt^j $$
in $\mF_p[t]$. It follows that the Euler characteristic of the categorical dimensions of the Koszul complex in degree $d$ is equal to the coefficient of $t^d$ in $(1+t)^{m+n}$. All these coefficients vanishing for $d\le p^r$ is equivalent to $p^{r+1}$ dividing $m+n$.
\end{proof}

As can be seen from the proof, it is in fact sufficient for the complex to be exact (or simply to have zero dimension Euler characteristic) in degrees $d=p^i$ for $i\le r$.

\begin{lemma}\label{LemExactSeq}
For any object $X$ in a tensor category $\cC$, the sequence
$$0\to\Lambda^3X\to X\otimes\Lambda^2X\to\Sym^{2}X\otimes X\to\Sym^3X\to 0$$
is exact. That is, the Koszul complex is always exact in degree 3.
\end{lemma}

\begin{proof}
First we prove that 
$$\Sym^{n-2}X\otimes\Lambda^2X\to\Sym^{n-1}X\otimes X\to\Sym^nX\to 0$$
is always exact for $n\ge 3$.

The kernel $K$ of the projection $\Sym^{n-1}X\otimes X\twoheadrightarrow\Sym^nX$ is given by the quotient of the subobjects of $X^{\otimes n}$ defining $\Sym^nX$ and $\Sym^{n-1}X\otimes X$, that is
\begin{align*}
K&=\frac{\sum_{i=1}^{n-1}X^{\otimes i-1}\otimes\Lambda^2X\otimes X^{\otimes n-i-1}}{\sum_{i=1}^{n-2}X^{\otimes i-1}\otimes\Lambda^2X\otimes X^{\otimes n-i-1}}\\
&=\frac{X^{\otimes n-2}\otimes\Lambda^2X}{\left(\sum_{i=1}^{n-2}X^{\otimes i-1}\otimes\Lambda^2X\otimes X^{\otimes n-i-1}\right)\cap\left(X^{\otimes n-2}\otimes\Lambda^2X\right)}\\
&=\frac{X^{\otimes n-2}\otimes\Lambda^2X}{\left(\sum_{i=1}^{n-3}X^{\otimes i-1}\otimes\Lambda^2X\otimes X^{\otimes n-i-3}\right)\otimes\Lambda^2X+\left(X^{\otimes n-3}\otimes\Lambda^3X\right)}\\
&=\frac{\Sym^{n-2}X\otimes\Lambda^2X}{\left(\Sym^{n-2}X\otimes\Lambda^2X\right)\cap\left(X^{\otimes n-3}\otimes\Lambda^3X\right)}.
\end{align*}
For $n=3$ we have $K=(X\otimes\Lambda^2X)/\Lambda^3X$ giving the exact sequence.
\end{proof}

\begin{prop}\label{PropNX6}
If $p\in\{2,3\}$, then $\mathbf N(X)$ is always a multiple of $p^2$.
\end{prop}
\begin{proof}
The Koszul complex in degree $2$ is obviously exact, so the conclusion for $p=2$ follows from Proposition~\ref{PropKoszul}.
For $p=3$, the claim follows from the combination of Proposition~\ref{PropKoszul} and Lemma~\ref{LemExactSeq}.
\end{proof}

\subsection*{Acknowledgement}

K. C.'s work was partially supported by the ARC grants DP210100251 and FT220100125. P. E.'s work was partially supported by the NSF grant DMS - 2001318. J. N.'s work was supported by an Australian Government RTP Scholarship.


\begin{thebibliography}
	{EGNO}




\bibitem[BE1]{BE1}  D.~Benson, P.~Etingof: Symmetric tensor categories in characteristic 2. Adv. Math. 351 (2019), 967--999.


\bibitem[BE2]{BE} D.~Benson, P.~Etingof: On cohomology in symmetric tensor categories in prime characteristic. Homology Homotopy Appl. 24 (2022), no.2, 163--193.

\bibitem[BEO]{BEO} D.~Benson, P.~Etingof, V.~Ostrik: New incompressible symmetric tensor categories in positive characteristic. Duke Math. J. 172 (2023), no. 1, 105--200.

\bibitem[Bo]{Bo} N.~Bourbaki: Groupes et alg\'ebres de Lie, chap. 4, 5 et 6, Hermann, Paris, 1968; chap. 7 et 8, Hermann, Paris, 1975.

\bibitem[CE]{CE} K.~Coulembier, P.~Etingof: $N$-spherical functors and tensor categories. Int. Math. Res. Not. IMRN (2024), no. 14, 10615--10649.

\bibitem[CEO1]{CEO1} K.~Coulembier, P.~Etingof, V.~Ostrik: On Frobenius exact symmetric tensor categories. With an appendix by A.~Kleshchev. Ann. of Math. (2) 197 (2023), no. 3, 1235--1279.



\bibitem[CEO2]{CEO2}  K.~Coulembier, P.~Etingof, V.~Ostrik: Asymptotic properties of tensor powers in symmetric tensor categories. Pure Appl. Math. Q. 20 (2024), no. 3, 1141--1179.

\bibitem[CEO3]{CEO3} K.~Coulembier, P.~Etingof, V.~Ostrik: Incompressible tensor categories. Adv. Math. 457 (2024), Paper No. 109935.

\bibitem[Co1]{Co} K.~Coulembier: Tannakian categories in positive characteristic. Duke Math. J. 169 (2020), no. 16, 3167-–3219.


\bibitem[Co2]{ComAlg} K.~Coulembier: Commutative algebra in tensor categories. arXiv:2306.09727.


\bibitem[CS]{CS} L.~Cagliero, F.~Szechtman:
The classification of uniserial $\mathfrak{sl}(2)\ltimes V(m)$-modules and a new interpretation of the Racah-Wigner $6j$-symbol. 
J. Algebra 386 (2013), 142-175

\bibitem[De1]{Del01} P.~Deligne: Cat\'egories tannakiennes. Progr. Math. 87 (1990), 111–195

\bibitem[De2]{Del02} P.~Deligne: Cat\'egories tensorielles. Mosc. Math. J. 2 (2002), no. 2, 227--248.	

\bibitem[DGNO]{DGNO} V. Drinfeld, S. Gelaki, D. Nikshych, V. Ostrik: On braided fusion categories I. Selecta Math. 16 (2010), no. 1, 1–119

\bibitem[DKS]{DKS}  T.~Dyckerhoff, M.~Kapranov, V.~Schechtman: N-spherical functors and categorification of Euler's continuants. arXiv:2306.13350.


\bibitem[Et]{Et} P.~Etingof:
Koszul duality and the PBW theorem in symmetric tensor categories in positive characteristic.
Adv. Math.327(2018), 128--160.

\bibitem[EGNO]{EGNO}P.~Etingof, S.~Gelaki, D.~Nikshych, V.~Ostrik:
Tensor categories. 
Mathematical Surveys and Monographs, 205. American Mathematical Society, Providence, RI, 2015. 

\bibitem[EHO]{EHO}P.~Etingof, N.~Harman, V.~Ostrik:
$p$-adic dimensions in symmetric tensor categories in characteristic $p$.
Quantum Topol. 9 (2018), no. 1, 119–140.

\bibitem[EO]{EO} P.~Etingof, V.~Ostrik:
On semisimplification of tensor categories. Representation theory and algebraic geometry, 3--35.
Trends Math.
Birkhäuser/Springer, Cham, 2022.


\bibitem[Ga]{Ga} A.~Garnier: Equivariant triangulations of tori of compact Lie groups and hyperbolic extension to non-crystallographic Coxeter groups. J. Algebra, 635 (2023), 527--576

\bibitem[GK]{GK} S.~Gelfand, D.~Kazhdan: Examples of tensor categories. Invent. Math., 109, no. 3 (1992), 595-617

\bibitem[GM]{GM} G.~Georgiev, O.~Mathieu: Fusion rings for modular representations of Chevalley groups. Contemp. Math., 175 (1994), 89-100

\bibitem[Gr]{Gr} J.~Gruber: Linkage and translation for tensor products of representations of simple algebraic groups and quantum groups. arXiv:2304.07796.

\bibitem[Hu]{Hu}
J.E.~Humphreys: Introduction to Lie Algebras and Representation Theory. Springer, 1972


\bibitem[Ja]{Jantzen}
J.C.~Jantzen:
Representations of algebraic groups. 
Second edition. Mathematical Surveys and Monographs, 107. American Mathematical Society, Providence, RI, 2003.

\bibitem[Le]{LiE}
M. Leeuwen: LiE, a software package for Lie group computations. (1994).


\bibitem[Os]{Os} V.~Ostrik: On symmetric fusion categories in positive characteristic. Selecta Math. (N.S.) 26 (2020), no. 3, Paper No. 36, 19 pp.

\bibitem[Sa]{Sa} S.~Sawin: Jones–Witten Invariants for Nonsimply Connected Lie Groups and the Geometry of the Weyl Alcove. Advances in Mathematics 165(1) (2002), 1-34

\bibitem[St]{St} D.~Stewart: On the minimal modules for exceptional Lie algebras: Jordan blocks and stabilisers. LMS J. Comput. Math. 19 (2016) 235–258.

\bibitem[TW]{TW} D.~Tubbenhauer, P.~Wedrich: Quivers for $\SL_2$ tilting modules. Represent. Theory 25 (2021), 440-480

\bibitem[Ve1]{Ve1} S.~Venkatesh: Harish-Chandra Pairs in the Verlinde Category in Positive Characteristic, International Mathematics Research Notices 2023(18) (2023), 15475–15536

\bibitem[Ve2]{Ve2} S.~Venkatesh: Representations of General Linear Groups in the Verlinde Category. arXiv:2203.03158.

 \end{thebibliography}
\end{document}